\documentclass[12pt, letterpaper]{article}

\include{macros}
\usepackage{amsmath,amssymb,amsfonts,amsthm,mathtools}
\usepackage{tikz,pgfplots}
\usetikzlibrary{arrows,decorations.pathreplacing}
\usepackage[bottom]{footmisc}
\usepackage{enumerate}
\usepackage{multirow}
\newcounter{Empty}[section]

\numberwithin{equation}{section}

\newcommand{\R}{\mathbb{R}}
\newcommand{\Z}{\mathbb{Z}}
\newcommand{\T}{\mathbb{T}}

\newcommand{\abs}[1]{\left|#1\right|}

\newcommand{\twonorm}[1]{\Vert#1\Vert_2}
\newcommand{\inftynorm}[1]{\Vert#1\Vert_\infty}

\newcommand{\Rm}[1]{{\rm #1}}
\newcommand{\F}{\mathcal{F}}

\theoremstyle{plain}
\newtheorem{theorem}{Theorem}[section]
\newtheorem{claim}{Claim}[section]

\newtheorem{lemma}{Lemma}[section]
\newtheorem{corollary}{Corollary}[section]

\theoremstyle{definition}

\newtheorem{remark}{Remark}[section]
\begin{document}
\renewcommand{\baselinestretch}{1.2}

\title{Asymptotic Seed Bias in Respondent-driven Sampling}
\author{Yuling Yan\thanks{Department of Operations Research and Financial Engineerging, Princeton University, Princeton, NJ 08544, USA;  E-mail: \texttt{yulingy@princeton.edu}.}, 
	Bret Hanlon\thanks{Department of Biostatistics and Medical Informatics, University of Wisconsin-Madison, Madison, WI 53726, USA; E-mail: \texttt{bret.hanlon@wisc.edu}.},
	Sebastien Roch\thanks{Department of Mathematics, University of Wisconsin-Madison, Madison, WI 53792, USA; E-mail: \texttt{roch@math.wisc.edu}.}, Karl Rohe\thanks{Department of Statistics, University of Wisconsin-Madison, Madison, WI 53706, USA; E-mail: \texttt{karlrohe@stat.wisc.edu}.}}

\date{}
\maketitle
\vspace*{-0.3 in}
\begin{abstract}
Respondent-driven sampling (RDS) collects a sample of individuals in a networked population by incentivizing the sampled individuals to refer their contacts into the sample. This iterative process is initialized from some seed node(s). Sometimes, this selection creates a large amount of seed bias. Other times, the seed bias is small. This paper gains a deeper understanding of this bias by characterizing its effect on the limiting distribution of various RDS estimators. Using classical tools and results from multi-type branching processes \citep{kesten1966additional}, we show that the seed bias is negligible for the Generalized Least Squares (GLS) estimator and non-negligible for both the inverse probability weighted and Volz-Heckathorn (VH) estimators. In particular, we show that (i) above a critical threshold, VH converge to a non-trivial mixture distribution, where the mixture component depends on the seed node, and the mixture distribution is possibly multi-modal.  Moreover, (ii) GLS converges to a Gaussian distribution independent of the seed node, under a certain condition on the Markov process. Numerical experiments with both simulated data and empirical social networks suggest that these results appear to hold beyond the Markov conditions of the theorems.
\end{abstract}

\sloppy

\tableofcontents{}
\newpage

\section{Introduction}
\label{sec_introduction}
Network sampling techniques, including web crawling, snowball sampling, and respondent-driven sampling (RDS), contact individuals in hard-to-reach populations by following edges in a social network.  This paper uses RDS as a motivating example \citep{heckathorn1997respondent}.  It is used by the Centers for Disease Control (CDC) and the Joint United Nations Programme on HIV/AIDS (UN-AIDS) to sample populations most at risk for HIV (injection drug users, sex workers, and men who have sex with men) \citep{HIVbehavioralSurveilance, johnston2013introduction}.  In the most recent survey of the literature \citep{white2015strengthening}, RDS had been applied in over 460 different studies, in 69 different countries.

An RDS sample is initialized with one or more ``seed individuals'' selected by convenience from the population. These individuals participate in the survey and are  incentivized to refer additional participants (often up to 3 or 5 participants) into the sample. This process iterates until reaching the target sample size or there are no referrals. All participants are incentivized to take a survey and an HIV test.  With this sample, we wish to estimate the proportion of individuals in the population that are HIV+. 

\begin{table}[b]
	\centering
	\caption{Summary of properties of IPW and GLS estimators.  In the columns, $m$ refers to the number of participants that the typical participant refers into the study and $\lambda_2$ is the second eigenvalue of the Markov transition matrix. }
	\label{table_previous}
	\begin{tabular}{|c|c|c|c|}
		\hline
		Result                        & Estimator & Low variance, i.e. $m<\lambda_2^{-2}$                                                 & High variance, i.e. $m>\lambda_2^{-2}$                                                     \\ \hline
		\multirow{2}{*}{Variance}     & IPW       & \begin{tabular}[c]{@{}c@{}}$O(n^{-1})$\\ \citep{rohe2019critical}\end{tabular}                                                    & \begin{tabular}[c]{@{}c@{}}$O(n^{2 \log_m \lambda_2})$\\ \citep{rohe2019critical}\end{tabular}                                  \\ \cline{2-4} 
		& GLS       & \multicolumn{2}{c|}{$O(n^{-1})$ \citep{roch2018generalized}}                                                                                                                       \\ \hline
		\multirow{2}{*}{Distribution} & IPW\&VH       & \begin{tabular}[c]{@{}c@{}}Asymptotically normal\\ \citep{li2017central}\end{tabular} & \begin{tabular}[c]{@{}c@{}}Non-trivial mixture \\ $[$Current paper$]$ \end{tabular} \\ \cline{2-4} 
		& GLS       & \multicolumn{2}{c|}{Asymptotically normal $[$Current paper$]$}                                                                                                        \\ \hline
	\end{tabular}
\end{table}

The Markov model for the RDS process has provided fundamental insight into RDS sampling \citep{salganik2004sampling, goel2009respondent, rohe2019critical}.  For example, nodes with more connections are more likely to be sampled \citep{levin2009markov}.  This creates bias and there are ways to adjust for it \citep{salganik2004sampling, volz2008probability}.   While the inverse probability weighted (IPW) estimator requires a normalizing constant that is unknown in practice, the Volz-Heckathorn (VH) estimator provides a way to estimate this normalizing constant \citep{volz2008probability}.  More recently, \cite{rohe2019critical} studied the variability of the IPW estimators and showed that there are two regimes (low variance and high variance).  This regime is determined by two parameters of the Markov process that is described in Section \ref{subsec_markovmodel}.  In brief,  let $\lambda_2$ be the second eigenvalue of the Markov transition matrix on the social network and let $m$ be the average number of referrals provided by each node.  When $m < \lambda_2^{-2}$, the variance of the IPW estimator decays at rate $n^{-1}$, where $n$ is the sample size.  However, when $m > \lambda_2^{-2}$, the variance of IPW decays at a slower rate.  Later, \cite{li2017central} showed that the VH and IPW estimators are asymptotically normal under the Markov model in the low variance regime. More recently, \cite{roch2018generalized} proposed a generalized least squares (GLS) estimator for the high variance regime and showed that the variance of this estimator is $O(n^{-1})$, even when $m>\lambda_2^{-2}$.  These previous results are summarized in Table \ref{table_previous}.

This paper studies the limit distribution of (i) the GLS estimator and (ii) the IPW estimator in the high variance regime.  These results also allow for the Volz-Heckathorn adjustment.  For technical reasons, our analysis of the GLS estimator is restricted to a special case of the Markov model that was first used to study RDS in \cite{goel2009respondent}.  

These technical results make many unrealistic assumptions which we discuss below.  In particular, the Markov model allows for resampling of individuals. The results are asymptotic in the sample size, while the population size is fixed. This creates extensive resampling.   Nevertheless, this model  provides fundamental insights into the properties of the estimators and these properties continue to hold under more realistic simulation models in Sections \ref{sec_sumulation} and \ref{section_analysis}.

\begin{figure}[t]
	\centering
	\includegraphics[height=11cm]{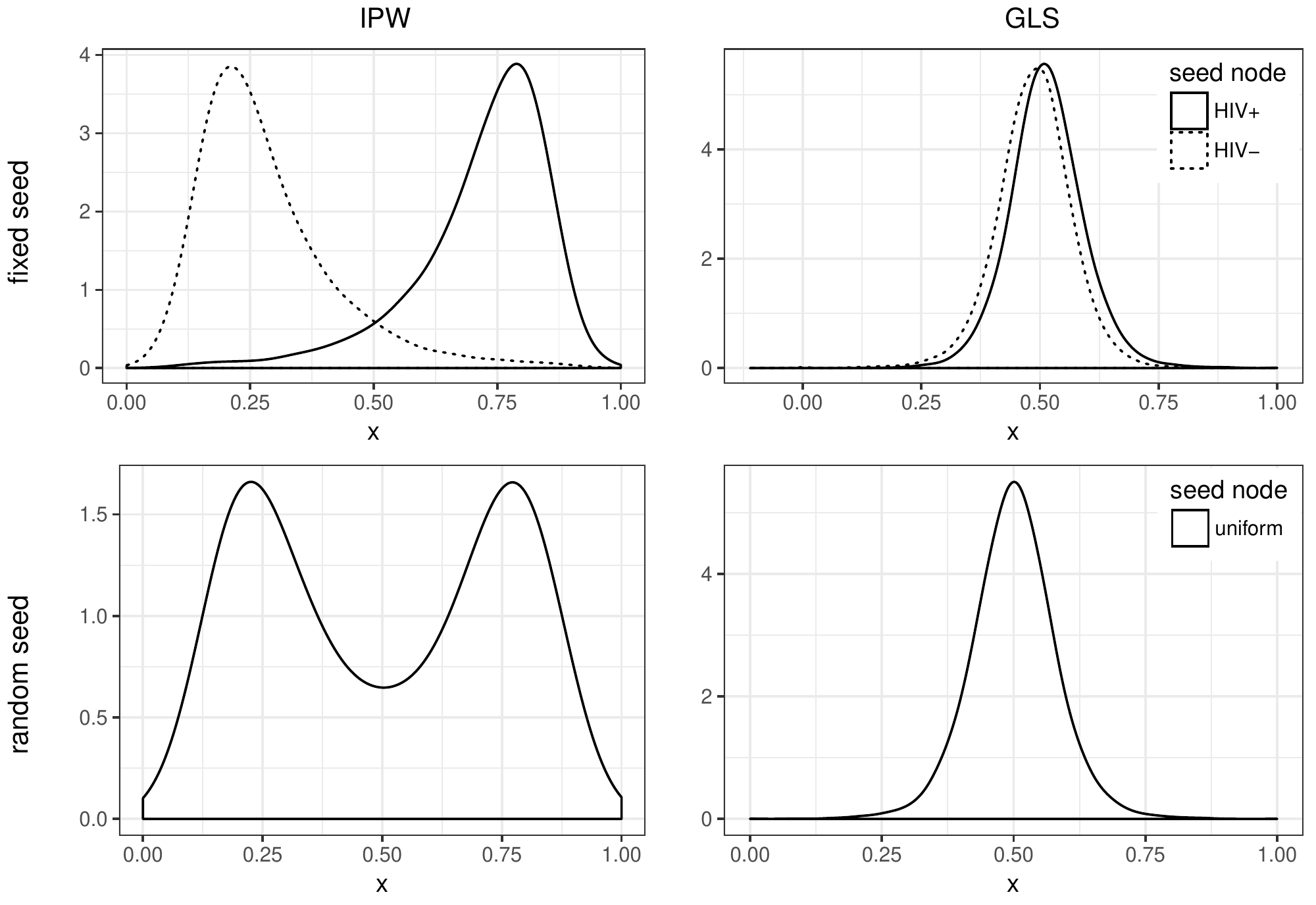}
	\caption{The model for this simulation is described in Section \ref{sec:motivate}. The two left panels show the distribution of sample proportion (i.e. the IPW estimator in this model). The two right panels show the distribution of GLS estimator. 
		Each panel in the top row has two curves corresponding to whether or not the seed node is HIV+.  The solid line gives the distribution of the estimator when the process is initialized with an HIV+ node.  The dashed line is initialized with an HIV- node.  In the bottom row, the seed participant is selected uniformly at random. This figure demonstrates how the limit distribution of the IPW estimator can have two modes which correspond to whether the seed is HIV+ or HIV-. Moreover, the figure suggests that the GLS estimator is asymptotically normal and the dependence on the seed node is negligible.}
	\label{fig_distributions_mean_gls}
\end{figure}

\subsection{A simple motivating example}\label{sec:motivate}
Here we consider a model studied in \cite{goel2009respondent}, which we refer to as the Blockmodel with $2$ blocks. In this example, the population that we wish to sample is equally divided into two groups: HIV+ and HIV-. The seed participant is selected from one of the two groups with equal probability. Each participant refers an iid number of offspring, generated from some offspring distribution. With probability $p$, the referred participant matches the HIV status of the participant that referred them.  With probability $1-p$, their statuses differ. Each referral is independent, conditional on the status of the referring participant. Using a sample generated in this way, we wish to estimate the proportion of the population that is HIV+ (in this case, the true proportion is $0.5$). 

Figure \ref{fig_distributions_mean_gls} displays a motivating simulation from this Blockmodel with $2$ blocks.  Each sample size is 1000 individuals, sampled from the Blockmodel with $p=.95$ and offspring distribution $1+\Rm{Binomial}(2,0.5)$.  For each sample of 1000, we construct both sample proportion (equivalent to the IPW estimator, see Section \ref{sec_ipw}) and GLS estimator.  This process is repeated 10000 times. Figure \ref{fig_distributions_mean_gls} displays a kernel density estimate of the resulting distribution. 

\subsection{Main contributions}
Many RDS papers discuss the ``bias from seed selection''.  Section \ref{subsec_sample_mean} shows that the IPW and VH estimators have a limit distribution and this limit distribution depends on where the process is initialized (i.e. the ``seed'' node).  If the seed node is randomized, then in simulations, the limit distribution of the IPW and VH estimators can have multiple modes, where each mode corresponds to a different set of initial conditions.  The limit results for the IPW and VH estimators highlight how, conditioned on the seed node, the bias of these estimators decays at the same rate as the variance. So, unconditional on the seed node, this can create multiple modes in the limit distributions of the IPW and VH estimators.  Similarly to classical results in multitype branching process theory \citep{kesten1966additional}, the exact limit distribution does not appear to have a concise and easily interpretable closed form. 

While the IPW and VH estimators are not asymptotically normal in the high variance regime, Section \ref{subsec_gls} shows that the GLS estimator is asymptotically normal in this regime and this limit distribution does not depend on where the process is initialized. This pair of results provides additional insight into the notions of ``bias'' and ``variance'' for network sampling.  In particular, the GLS estimator is the linear estimator with the smallest variance and that measure of variance includes the variability that comes from selecting the seed node (i.e. from the stationary distribution of the Markov process).  Hence, it adjusts for the seed selection. Another way of saying this is that the GLS estimator reduces ``the bias from seed selection''.  This blurring of the divide between ``variance" and ``bias from seed selection'' highlights one potential problem of conditioning on the seed node in a bootstrap resampling procedure \citep{baraff2016estimating};  in the high variance regime, conditioning on the seed node removes a large source of variability in the VH estimator. 

\section{Background and notation}
\label{section_background} 
This section (i) defines the Markov model, (ii) illustrates how this model is particularly tractable when the underlying network is a Blockmodel \citep{White1976Social}, and (iii) defines the IPW, VH, and GLS estimators.

\subsection{Markov model}
\label{subsec_markovmodel}
The Markov model consists of (1) a social network represented as a graph, (2) a Markov transition matrix on the nodes of the graph, (3) a referral tree to index the Markov process on the graph, and finally, (4) a node feature defined for each node in the graph.  Each of these are defined below.

The results in this paper allow for an undirected, weighted graph.  Let $G=(V,E)$ be a graph with vertex set $V=\{1,\ldots,N\}$ containing the people and edge set $E=\{(i,j):i,j\in V \ \Rm{are} \ \Rm{connected} \}$ containing the friendships. 
Let $w_{ij}$ be the weight of the edge $(i,j)\in E$.  For notational convenience, define $w_{ij}=0$ if $(i,j)\notin E$. If the graph is unweighted, define $w_{ij}=1$ for all $(i,j)\in E$. Throughout this paper, the graph is undirected (i.e. $w_{ij}=w_{ji}$ for all pairs $(i,j)$). Define the degree of node $i$ as $\mathrm{deg}(i)=\sum_{j} w_{ij}$ and the volume of the graph as $\mathrm{vol}(G)=\sum_{i} \mathrm{deg}(i)$. For simplicity, $i\in G$ is used synonymously with $i \in V$. Define the Markov transition matrix $\bP\in\R^{N\times N}$ as
\begin{equation}
\label{eq_tran_matrix}
P_{ij}=\frac{w_{ij}}{\mathrm{deg}(i)}.
\end{equation}
Since $G$ is undirected, $\bP$ is a reversible Markov transition matrix with a stationary distribution $\bpi: G\rightarrow \R$ with $\pi(i)=\mathrm{deg}(i)/\mathrm{vol}(G)$. 

The referral tree is a rooted tree, i.e. a connected graph with $n$ nodes, no cycles, and a vertex $0$. This tree, $\T$, can be random (a Galton-Watson tree with expected offspring number $m$) or nonrandom (an $m$-tree, where each node has exactly $m$ offspring). If $\T$ is randomly generated, then the Markov process is conditioned on the tree. For simplicity, $\sigma\in\T$ is used synonymously with $\sigma$ belonging to the vertex set of $\T$. The seed participant is the root vertex $0$ in $\T$. For each non-root node $\sigma\in\T$, denote $p(\sigma)\in\T$ as the parent of $\sigma$ (i.e. the node one step closer to the root).

Assume that the nodes are sampled with a Markov process that is indexed by $\T$: each node $\sigma\in\T$ corresponds to an individual $X_\sigma$ sampled from the population $G$, and an edge $(\sigma,\tau)$ of $\T$ denotes that the sampled individual $X_\sigma$ referred the individual $X_\tau$ into the sample. Mathematically, let $\{X_\sigma^{(\cdot)}: \sigma\in\T\}$ be a tree-indexed Markov process on the individuals from the social network $G$:
\begin{equation*}
\PP(X_\sigma^{(\cdot)}=j \mid X_{p(\sigma)}^{(\cdot)}=i, X_\tau^{(\cdot)}:\tau \in \mathscr{D}(\sigma)^c)=\PP(X_\sigma^{(\cdot)}=j\mid X_{p(\sigma)}=i)=P_{ij},
\end{equation*}
where $\mathscr{D}(\sigma) \subset \T$ denotes the set of $\sigma$ and all its descendants in $\T$. The superscript $(\cdot)$ indicates the initial condition: if the superscript is some $i\in G$, $X_0$ is initialized from $i$; if the superscript is some distribution $\bnu: G\to \R$ (e.g. the stationary distribution $\bpi$ of $\bP$), $X_0$ is initialized from $\bnu$. When the initial state does not matter, we leave off the superscript. Following \cite{benjamini1994markov}, we call this process a $(\T,\bP)$-walk on $G$.

In a special case, $\T$ can be the chain graph ($ 0 -  1 -2 -3- \dots $); this results in the model being a Markov chain. Just as a chain graph  indexes a Markov chain, the graph $\T$ provides the indexing in this model. For simplicity, $\sigma\in\T$ is used synonymously with $\sigma$ belonging to the vertex set of $\T$. The seed participant is root vertex $0$ in $\T$. For each non-root node $\sigma\in\T$, denote $p(\sigma)\in\T$ as the parent of $\sigma$ (i.e. the node one step closer to the root). Assume that the nodes are sampled with a Markov process that is indexed by $\T$.

For each node $i\in G$, let $y(i)$ denote some characteristic of this node, for example whether $i$ is HIV+ or HIV-.  Sometimes we regard $\by$ as a vector in $\R^N$, where $N$ is the number of nodes in $G$. We want to estimate the population average $\mu_{\mathsf{true}}=\sum_{i \in G} y(i)/N$
by the RDS sample $\{y(X_\sigma):\sigma\in\T\}$. 

\subsection{A special case: Blockmodel} 
\label{section_rank_k}
Consider $G$ as coming from a Blockmodel with $k$ blocks \citep{White1976Social}. That is, each node $i\in G$ is assigned to a block with $b(i) \in \{1,\ldots,k\}$, where each block $j$ contains $N/k$ nodes. If $b(i) = b(j)$, then $w_{i\ell} = w_{j \ell}$ for all $\ell \in \{1, \dots N\}$. Further suppose that if $b(i)=b(j)$, then $y(i)=y(j)$. The Stochastic Blockmodel \citep{Holland1983Stochastic} is derived from this model.

The idea behind a Blockmodel with $k$ blocks is clear: people in the same block share the same feature and the same friendship patterns. \cite{goel2009respondent} studied RDS with this model. The motivating example in Section \ref{sec_introduction} also uses a Blockmodel with $2$ blocks. 

Let $\mathcal{W}\in\R^{k\times k}$ denote the weight matrix between blocks, where $\mathcal{W}_{b(i),b(j)}=w_{ij}$.
Define the corresponding Markov transition matrix between blocks $\mathcal{P}\in \R^{k\times k}$ from $\mathcal{W}$ similarly to \eqref{eq_tran_matrix}. Since $\mathcal{W}$ is symmetric, $\mathcal{P}$ is reversible. 

Let $\{B_\sigma^{(\cdot)}: \sigma\in\T\}$ denote a Markov process indexed by $\T$, where the state space is the block labels $\{1,\ldots,k\}$ and the transition matrix is $\mathcal{P}$. 
The superscript of $B_\sigma^{(\cdot)}$ indicates the initial state of $B_0$ and is in correspondence with the initial state $X_0$ of the Markov process over $G$: if $X_0$ is initialized at $i\in G$, $B_0$ is initialized at $z(i)$ and the superscript is $z(i)$; if $X_0$ is initialized from any distribution $\bnu: G\to\R$, $B_0$ is initialized from the distribution $\bmu: \{1,\ldots,k\}\to\R$ with $\mu_j=\sum_{i\in G:z(i)=j} \nu_i$. For any $\{\sigma_{i_1}, \ldots, \sigma_{i_s}\}\subset\T$ and $b_{i_1},\ldots,b_{i_s}\in\{1,\ldots,k\}$,
\begin{equation}
\label{eq:blockmodel}
\PP(B_{\sigma_{i_1}}^{(\cdot)}=b_{i_1},\ldots,B_{\sigma_{i_s}}^{(\cdot)}=b_{i_s})=\PP(b(X_{\sigma_{i_1}}^{(\cdot)})=b_{i_1},\ldots,b(X_{\sigma_{i_s}}^{(\cdot)})=b_{i_s}).
\end{equation}
The proof of \eqref{eq:blockmodel} is in Appendix \ref{section_proof_eq2}. So $\{B_\sigma^{(\cdot)}: \sigma\in\T\}$ is equal in distribution to $\{b(X_\sigma^{(\cdot)}): \sigma\in\T\}$. Instead of studying the Markov process $\{X_\sigma^{(\cdot)}: \sigma\in\T\}$ in Section \ref{subsec_markovmodel}, we study the Markov process $\{B_\sigma^{(\cdot)}: \sigma\in\T\}$. Intuitively, the original process $\{X_\sigma^{(\cdot)}: \sigma\in\T\}$ keeps track of the individuals while $\{B_\sigma^{(\cdot)}: \sigma\in\T\}$ keeps track of some feature of the individuals.
This time the node feature $\by\in\R^N$ is replaced by the block feature $\bb\in \R^k$ and the Markov transition matrix is replaced by the Markov transition matrix between blocks $\mathcal{P}\in\R^{k\times k}$.

The Blockmodel is a special case of the Markov model in Section \ref{subsec_markovmodel}. In this paper, Theorem \ref{thm_mut}, Corollary \ref{cor_decay} and \ref{cor_VH} apply to the Markov model. Theorem \ref{thm_gls} and Corollary \ref{cor_glsVH} only apply to the Blockmodel with $2$ blocks.

\subsection{Estimators}
\label{sec_estimators}
Denote $\EE_{\bpi} (y)=\sum_{i}\pi(i) y(i)$. The theoretical results in this paper  study two estimators defined in this section. They are unbiased estimators of $\EE_{\bpi} (y)$. When applying inverse probability weighting (in Section \ref{sec_ipw}), these estimators become unbiased estimators of $\mu_{\mathsf{true}}$ instead. Further, the VH adjustment provides a way to estimate the inverse probability weights. 

\paragraph{Sample average}
\label{subsec_sampleaverage}
The RDS sample average is
\begin{equation}
\hat{\mu}^{(\cdot)}=\frac{1}{n} \sum_{\sigma\in\T} y(X_\sigma^{(\cdot)}).
\end{equation}
When $X_0$ is initialized from $\bpi$, $\hat{\mu}^{(\bpi)}$ is an unbiased estimator of $\EE_{\bpi} (y)$. When $X_0$ is initialized from $i\in G$, $\hat{\mu}^{(i)}$ is an asymptotically unbiased estimator of $\EE_{\bpi} (y)$ (see Claim \ref{lem:trivial}).

\paragraph{GLS estimator}
\label{subsec_glsestimator}
\cite{roch2018generalized} proposed generalize least squares (GLS) in RDS to reduce the variance, particularly  in the high variance regime.  The GLS estimator is the weighted average 
\begin{equation}
\hat{\mu}^{(\cdot)}_{\mathsf{GLS}}=\sum_{\sigma\in\T} w_\sigma^\ast y(X^{(\cdot)}_\sigma)
\end{equation}
where $\bw^\ast$ minimizes the variance of the weighted average initialized from $\bpi$
\begin{equation}
\label{eq:glsdef}
\bw^\ast=\arg\min_{\bw} \,\mathsf{Var}\left(\sum_{\sigma\in\T} w_\sigma y(X^{(\bpi)}_\sigma)\right) \quad s.t. \quad \sum_{\sigma\in\T} w_\sigma=1.
\end{equation}
When $X_0$ is initialized from $\bpi$, $\hat{\mu}^{(\bpi)}_{\mathsf{GLS}}$ is an unbiased estimator of $\EE_{\bpi} (y)$. When $X_0$ is initialized from $i\in G$, $\hat{\mu}^{(i)}_{\mathsf{GLS}}$ is an \emph{asymptotically} unbiased estimator of $\EE_{\bpi} (y)$ (see Theorem \ref{thm_gls}).

\subsection{Inverse probability weighting}
\label{sec_ipw}
In general $\mu_{\mathsf{true}}\neq\EE_{\bpi} (y)$. So $\hat{\mu}$ and $\hat{\mu}_{\mathsf{GLS}}$ are biased estimators for $\mu_{\mathsf{true}}$. 
Inverse probability weighting can adjust for this bias. Define
$y^{\bpi}(i)=y(i)/(N\pi(i))$. The
IPW estimator and GLS estimator with IPW adjustment are the sample average and the GLS estimator of $y^{\bpi}(X_\sigma)$'s:
$$\hat{\mu}_{\mathsf{IPW}}=\frac{1}{n} \sum_{\sigma\in\T} y^{\bpi}(X_\sigma)= \frac{1}{n} \frac{\mathrm{vol}(G)}{N} \sum_{\sigma \in \T} \frac{y(X_\sigma)}{\mathrm{deg}(X_\sigma)}, \ \Rm{and}$$
$$\hat{\mu}_{\mathsf{IPW,GLS}}=\sum_{\sigma\in\T} w^{\bpi}_\sigma y^{\bpi}(X_\sigma)= \frac{\mathrm{vol}(G)}{N} \sum_{\sigma \in \T} w^{\bpi}_\sigma \frac{y(X_\sigma)}{\mathrm{deg}(X_\sigma)}.$$
When $X_0$ is initialized from the stationary distribution $\bpi$, they are unbiased estimates of $\mu_{\mathsf{true}}$.
However, computing these two estimators requires the average node degree $\mathrm{vol}(G)/N$, which is typically not available in practice. 

The popular VH estimator replaces $\mathrm{vol}(G)/N$ in the IPW estimator with the harmonic mean of the degrees of the RDS samples \citep{volz2008probability}. Define 
$$ H^{-1}=\frac{1}{n} \sum_{\sigma \in \T} \frac{1}{\mathrm{deg}(X_{\sigma})}, \qquad \hat{\pi}(i)=H^{-1} \mathrm{deg}(i), \qquad y^{\hat{\bpi}}(i)=\frac{y(i)}{\hat{\pi}(i)}.$$
The VH estimator is the sample average of $y^{\hat{\bpi}}(X_\sigma)$'s. The GLS estimator with VH adjustment uses a similar reweighting, but replaces $\mathrm{vol}(G)/N$ with a GLS estimate of $\EE_{\bpi} (1/\mathrm{deg}(i))$ \citep{roch2018generalized}.

The VH estimator and GLS estimator with VH adjustment are two asymptotically unbiased estimators of $\mu_{\mathsf{true}}$ under the $(\T,\bP)$-walk on $G$. Theorem \ref{thm_mut} and \ref{thm_gls} study the limit distribution of the sample average and GLS estimator. By a simple transformation (defining a new node function $y^{\bpi}(i)=y(i)/(N\pi(i))$), these results can also be applied to the IPW estimator and the GLS estimator with IPW adjustment. Corollary \ref{cor_VH} and \ref{cor_glsVH} extend these results to the VH estimator and GLS estimator with VH adjustment. 

\subsection{Additional notation}
For two sequences $a_n$ and $b_n$, define the following notation: (i) $a_n=O(b_n)$ if and only if $\abs{a_n}$ is bounded above by $b_n$ (up to constant factor) asymptotically, i.e. $\exists k>0, \ \exists n_0, \ \forall n>n_0, \abs{a_n}\leq kb_n$.
(ii) $a_n=\Theta(b_n)$ if and only if $a_n$ is bounded both above and below by $b_n$ (up to constant factors) asymptotically, i.e.
$\exists k_1>0, \ \exists k_2>0, \ \exists n_0, \ \forall n>n_0, \  k_1 b_n\leq a_n \leq k_2 b_n$.

\section{Main results}
\label{sec_main}
This section shows that, after proper scaling, the GLS estimator and the sample average both have a limit distribution. For GLS, the limit distribution is a normal distribution. For the sample average, on the other hand, the limit distribution is a non-trivial mixture distribution, where the mixture component is determined by the seed node.  This mixture distribution can be multi-modal as illustrated in Figure \ref{fig_distributions_mean_gls}. These results can be further extended to the GLS estimator with VH adjustment and to the VH estimator respectively.

We will need the following standard lemma (e.g.~\cite[Lemma 12.2]{levin2009markov}) which provides the eigendecomposition of the Markov transition matrix $\bP$.

\begin{lemma}
\label{lem_eigendecomposition}
	Let $\bP$ be a reversible Markov transition matrix on the nodes in $G$ with respect to the stationary distribution $\bpi$.  The eigenvectors of $\bP$, denoted as $\bbf_1, \dots, \bbf_{N}$, are real valued functions of the nodes $i \in G$ and orthonormal with respect to the inner product 
	\begin{equation} \label{def:inner}
	\langle \bbf_a, \bbf_b \rangle_{\bpi} = \sum_{i \in G} f_a(i) f_b(i) \pi(i).
	\end{equation}
	If $\lambda$ is an eigenvalue of $\bP$, then $|\lambda|\le 1$.  The eigenfunction $\bbf_1$ corresponding to the eigenvalue $1$ can be taken to be the constant vector $\bm{1}$.
\end{lemma}

Assume that the eigenvalues of $\bP$ are $$\abs{\lambda_1}\geq\abs{\lambda_2}\geq\cdots\geq\abs{\lambda_N}.$$  
Since it is a Markov transition matrix, its largest eigenvalue is $\lambda_1=1$.
Let $\bbf_i$ be the eigenvector corresponding to $\lambda_i$, normalized as in Lemma \ref{lem_eigendecomposition}. The eigenvector $\bbf_1$ corresponding to $\lambda_1$ is taken to be the constant vector $\bm{1}$. Expanding the node feature $y\in\R^N$ in the eigenbasis yields
\begin{equation}
\label{eq_eigendecompose}
\by=\sum_{j=1}^{N} \langle \by,\bbf_j \rangle_{\bpi} \bbf_j.
\end{equation} 

\subsection{Results for the sample average and the IPW and VH estimators}
\label{subsec_sample_mean}
This section shows that the sample average, IPW and VH estimators have a limit distribution and that this limit distribution in fact depends on where the process is initialized (i.e. the ``seed'' node).

For each node $\sigma\in\T$, let $\abs{\sigma}$ be the distance of $\sigma$ from the root $0$. Define $\{X_\sigma:\sigma\in\T,\abs{\sigma}=t\}$ as the individuals in the $t$-th generation of the sample. Denote the sample average up to generation $t$ as $\hat{\mu}_t$. Superscripts on $\hat \mu$ will denote how $X_0$ is initialized. 

Theorem \ref{thm_mut} studies the limit distribution of the sample average $\hat{\mu}_t^{(i)}$. Recall that the sample average of RDS samples is $\hat{\mu}^{(\cdot)}=n^{-1} \sum_{\sigma\in\T} y(X_\sigma^{(\cdot)}).$

\begin{theorem}
	\label{thm_mut}
	Assume the eigenvalues of the transition matrix $\bP$ are 
	\begin{equation}
	\label{eq_eigenvalue}
	1=\lambda_1>\lambda_2>\abs{\lambda_3}\geq\cdots\geq\abs{\lambda_N}.
	\end{equation}
	Assume $\T$ is an $m$-tree. When $m>\lambda_2^{-2}$, there exist a random variable $X^{(i)}\in L^2$ such that
	\begin{equation}
	\label{eq_mu_converge}
	\lambda_2^{-t}\left[\hat{\mu}_t^{(i)}-\EE_{\bpi} (y)\right] \rightarrow X^{(i)}
	\end{equation}
	almost surely and in $L^2$ as $t\to\infty$, and
	\begin{equation}
	\label{eq_different_expectation}
	\EE X^{(i)}=\frac{(m-1)\lambda_2}{m\lambda_2-1} 
	\,\langle \by,\bbf_2 \rangle_{\bpi} 
	\,f_2(i).
	\end{equation}
	Moreover, if $\langle \by, \bbf_2 \rangle_{\bpi}\neq 0$, then $\mathsf{Var}(X^{(i)})>0$ for any $i=1,\ldots,N$.
\end{theorem}

Note that the result is based on the technicial condition that $\T$ is an $m$-tree. The simulations in Section \ref{sec_sumulation} suggest that the result still holds when $\T$ is a Galton-Watson tree. Condition \eqref{eq_eigenvalue} in Theorem \ref{thm_mut} can be weakened to
\begin{equation*}
1=\lambda_1>\lambda_2=\cdots=\lambda_k>\abs{\lambda_{k+1}}\geq\cdots\geq\abs{\lambda_N}, 
\end{equation*}
but the statement of the conclusion becomes more involved.  
See Remark \ref{remark:generalize_eig} for a complete statement.


Using the above result, we can study how the bias and variance of the sample average decays, conditioned on the seed node. 

\begin{corollary}
	\label{cor_decay}
	Assume the conditions of Theorem \ref{thm_mut} hold. 
	\begin{enumerate}
		\item When $\langle \by,\bbf_2 \rangle_{\bpi}\neq 0$ and $f_2(i)\neq 0$,  the bias of $\hat{\mu}_t^{(i)}$ decays like
		\begin{equation}
		\left[\EE (\hat{\mu}_t^{(i)})-\EE_{\bpi} (y)\right]^2 = \Theta(\lambda_2^{2t}).
		\end{equation}
		\item When $\langle \by,\bbf_2 \rangle_{\bpi}\neq 0$, the variance of $\hat{\mu}_t^{(i)}$ decays like
		\begin{equation}
		\mathsf{Var} (\hat{\mu}_t^{(i)}) = \Theta(\lambda_2^{2t}).
		\end{equation}
	\end{enumerate}
\end{corollary}

When $X_0$ is initialized from $\bpi$, $\hat{\mu}_t^{(\bpi)}$ is an unbiased estimator of $\mu_{\mathsf{true}}$. By \eqref{eq_different_expectation}, for $i,j$ such that $f_2(i)\neq f_2(j)$, the limit distributions of $\lambda_2^{-t}\hat{\mu}_t^{(i)}$ and $\lambda_2^{-t}\hat{\mu}_t^{(j)}$ are different because $X^{(i)}$ and $X^{(j)}$ have different expectations. Thus the limit distribution of $\lambda_2^{-t}\hat{\mu}_t^{(\bpi)}$ is a non-trivial mixture.  The motivating example in the introduction illustrates this mixture. It is further explored with the simulation in Section \ref{sec_sumulation}.


Theorem \ref{thm_mut} studies the limit distribution of the sample average. Using the transformation discussed in Section \ref{sec_ipw}, the result also applies to the IPW estimator. Denote the VH estimator up to generation $t$ as $\hat{\mu}_{\mathsf{VH},t}$. The following corollary extends the result to the VH estimator.

\begin{corollary}
	\label{cor_VH}
	Under the conditions of Theorem \ref{thm_mut}, there exists a random variable $\tilde{X}^{(i)}\in L^2$ such that $$\lambda_2^{-t}\left[\hat{\mu}^{(i)}_{\mathsf{VH},t}-\mu_{\mathsf{true}}\right]\to \tilde{X}^{(i)}$$ 
	almost surely, and
	\begin{equation*}
	\EE \tilde{X}^{(i)}=\EE_{\bpi}(y^\prime)^{-1}\frac{(m-1)\lambda_2}{m\lambda_2-1} 
	\,\langle \by^{\prime\prime},\bbf_2 \rangle_{\bpi}\, 
	f_2(i),
	\end{equation*}
	where $y'(j)=\mathrm{deg}(j)^{-1}$ and $y''(j)=y(j)/\mathrm{deg}(j)$. Moreover, if $\langle \by^{\prime\prime},\bbf_2 \rangle_{\bpi} \neq 0$, then $\mathsf{Var}(\tilde{X}^{(i)})>0$ for any $i=1,\ldots,N$.
\end{corollary}
Similarly, when $X_0$ is initialized from $\bpi$, the limit distirbution of $\hat{\mu}_{\mathsf{VH},t}^{(\bpi)}$ is a non-trivial mixture of the limit distributions of $\hat{\mu}_{\mathsf{VH},t}^{(i)}$ for all $i\in G$.

\subsection{Results for the GLS estimator}
\label{subsec_gls}
For the GLS estimator, the two right panels of Figure \ref{fig_distributions_mean_gls} suggest that the estimator is not sensitive to the initial distribuiton of $X_0$. This section shows that the GLS estimator is asymptotically normal with parameters that do not depend on the initial distribution of $X_0$.

Given the referral tree $\T$, define the covariance matrix $\bSigma\in\R^{n\times n}$ as
\begin{equation*}
\bSigma_{\sigma,\tau}=\mathsf{Cov} (y(X_\sigma),y(X_\tau))
\end{equation*} 
for any $\sigma, \tau\in\T$, where $n$ is the number of nodes in $\T$. According to \cite{roch2018generalized}, $\bw^\ast$ in \eqref{eq:glsdef} is given by  
\begin{equation}
\label{eq_gls_weight}
\bw^\ast=(\bx^\top \textbf{1})^{-1} \bx^\top, \quad \text{where} \quad \bSigma \bx=\bm{1}.
\end{equation}
Here $\bx$ is the vectorization of the RDS sample $\{X_\sigma^{(\cdot)}: \sigma\in\T\}$.
For the Blockmodel with $2$ blocks, the GLS estimator admits a closed-form expression:
\begin{equation}
\label{eq_gls_expression}
\hat{\mu}_{\mathsf{GLS}}=\sum_{\sigma\in\T} \frac{1-\lambda_2 (\mathrm{deg}(\sigma)-1)}{n(1-\lambda_2(1-\frac{2}{n}))} y(X_\sigma),
\end{equation}
where $\lambda_2$ is the second eigenvalue of the Markov transition matrix between blocks and $\mathrm{deg}(\sigma)$ is the degree of $\sigma \in \T$. 

Let $\hat{\mu}_{\mathsf{GLS},t}$ be the GLS estimator of the RDS samples up to generation $t$. Based on \eqref{eq_gls_expression}, the following theorem establishes the asymptotic normality of the GLS estimator. 

\begin{theorem}
	\label{thm_gls}
	Consider the Blockmodel with $2$ blocks on
	an $m$-tree $\T$. Assume $\abs{\lambda_2}<1$. Then, for any initial distribution $\bnu$ of $X_0$,
	\begin{equation}
	\sqrt{n_t}\left[\hat{\mu}_{\mathsf{GLS},t}^{(\bnu)}-\EE_{\bpi} (y)\right] \to \cN\left(0,\frac{1+\lambda_2}{1-\lambda_2}\,\mathsf{Var}_{\bpi}(y)\right).
	\end{equation}
	in distribution as $t \to \infty$, where $\mathsf{Var}_{\bpi}(y)=\EE_{\bpi}(y^2)-(\EE_{\bpi}(y))^2$ and $n_t=1+m+\cdots+m^t$ is the number of RDS samples up to generation $t$. 
\end{theorem}

Theorem \ref{thm_gls} shows that the GLS estimator is asymptotically normal \emph{both in the low variance and high variance regimes}. Note that the result is based on \eqref{eq_gls_expression} and the technical condition that $\T$ is an $m$-tree. The simulations in Section \ref{sec_sumulation} suggest that the asymptotic normality of the GLS estimator still holds when $\T$ is a Galton-Watson tree, or the model is no longer a Blockmodel with $2$ blocks.

Theorem \ref{thm_gls} studies the limit distribution of the GLS estimator. Using the transformation discussed in Section \ref{sec_ipw}, the result also applies to the GLS estimator with IPW adjustment. Denote the GLS estimator with VH adjustment of RDS samples up to generation $t$ as $\hat{\mu}^{(\cdot)}_{\mathsf{GLS,VH},t}$. The following corollary extends the result to the GLS estimator with VH adjustment.

\begin{corollary}
	\label{cor_glsVH}
	Under the conditions in Theorem \ref{thm_gls}, for any initial distribution $\bnu$ of $X_0$,
	$$\sqrt{n_t}\left[\hat{\mu}_{\mathsf{GLS,VH},t}^{(\bnu)}-\mu_{\mathsf{true}}\right] \xrightarrow{d} \cN\left(0,\frac{1+\lambda_2}{1-\lambda_2}\,\EE_{\bpi}(y')^{-2}
	\,\mathsf{Var}_{\bpi}(y^{\prime\prime})\right).$$
	where $y'(i)=\mathrm{deg}(i)^{-1}$ and $y''(i)=y(i)/\mathrm{deg}(i)$.
\end{corollary}

\section{Simulation studies}\label{sec_sumulation}
In this section, data are simulated from a
Blockmodel with $2$ or $3$ blocks.  As stated in Section \ref{section_rank_k}, a Blockmodel with $k$ blocks consists of a reversible transition matrix $\mathcal{P}\in\R^{k\times k}$ between blocks, block feature $\by\in\R^k$, and a referral tree $\T$. In this specification, the block feature $\by$ is assumed to be centralized, so that $\EE_{\bpi} (y)=0$. For a Blockmodel with $2$ blocks, let
\begin{equation*}
\mathcal{P}=\Big(\begin{matrix}
p & 1-p \\ 1-q & q
\end{matrix}\Big).
\end{equation*}
denote the transition matrix between 2 blocks.
The second eigenvalue of $\mathcal{P}$ is $\lambda_2=p+q-1$. 

In the simulation settings below, the block feature is given prior to centralization. In fact, all of the $2$-Blockmodels use $\by=(1,0)^\top$ and the $3$-Blockmodels use $\by=(0,1,2)^\top$ . All of the experiments are based on 5000 simulated datasets.

\subsection{Sample average} \label{sec:sim_average}
Here we consider the behavior of the sample average $\hat{\mu}_t$ in the high variance regime $m>\lambda_2^{-2}$. In this setting, the asymptotic distribution of $\lambda_2^{-t}\hat{\mu}_t^{(\bpi)}$ is no longer normal, unlike the low variance regime. Instead, its asymptotic distribution is a mixture of the distributions of $\lambda_2^{-t}\hat{\mu}_t^{(i)}$ for all $i\in G$.

The simulation is performed on two different Blockmodels with $2$ blocks. We consider a balanced model with $p=q=.95$  and an unbalanced model with   $p=0.95$ and $q=0.85$. For both models, $\T$ is a Galton-Watson tree with offspring distribution $1+\Rm{Binomial}(2,1/2)$.  Under these settings, $m>\lambda_2^{-2}$ for both models. Figure \ref{fig_sample mean for two rank-2 models} displays the results of the experiment with $t=50$. 


\begin{figure}[t]
	\centering
	\includegraphics[height=11cm]{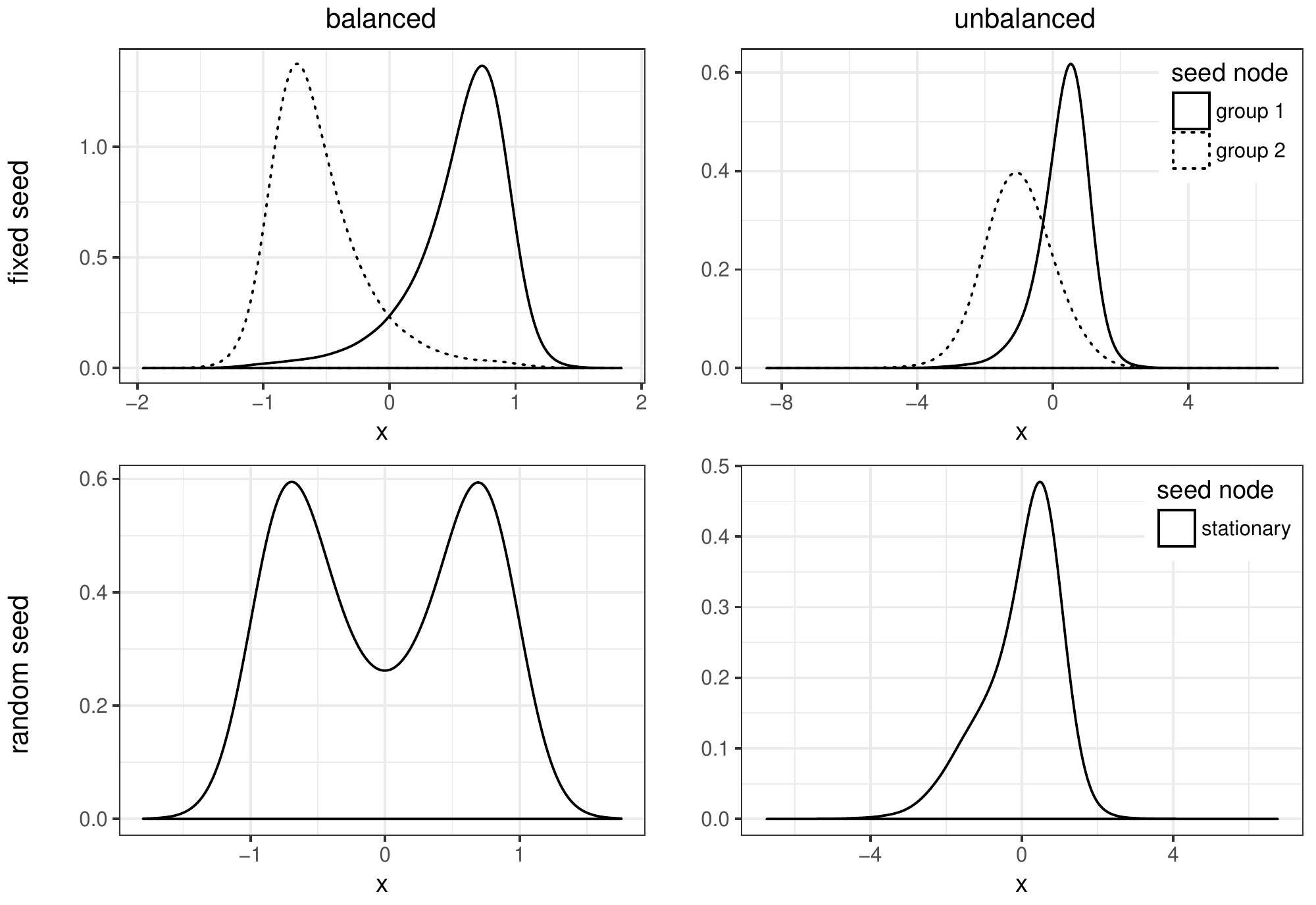}
	\caption{Kernel density estimates of $\lambda_2^{-t}\hat{\mu}_t$ for balanced (the left panels) and unbalanced (the right panels) Blockmodel with $2$ blocks over $5000$ replicates.  For each scenario, the top panel corresponds to the case when $X_0$ is initialized from group 1 (the solid curve) and group 2 (the dashed curve), the lower panel corresponds to the case when $X_0$ is initialized from the stationary distribution.}
	\label{fig_sample mean for two rank-2 models}
\end{figure}

\subsection{GLS estimator}\label{sec:sim_gls}
Here we consider the behavior of the GLS estimator in both the low and high variance regimes. The first experiment corroborates the result of Theorem \ref{thm_gls}, namely that the GLS estimator is asymptotically normal in both variance regimes. The simulation is performed on two different Blockmodels with $2$ blocks. In the first model $(p,q) = (0.95, 0.85)$; in the second model $(p,q) = (0.8, 0.7)$. For both models, $\T$ is a $2$-tree. Under these settings, $m>\lambda_2^{-2}$ for the first model and $m<\lambda_2^{-2}$ for the second model. The two quantile-quantile plots in Figure \ref{fig_gls_2tree} correspond to the two models. It appears that the distribution of the GLS estimator gets closer to the normal distribution as the sample size increases.

The second experiment suggests that the asymptotic normality of GLS estimator extends beyond the conditions in Theorem \ref{thm_gls}. We consider a two-block model with $(p,q) = (0.8, 0.7)$ and a three-block model, where the transition matrix between the blocks is
\begin{equation*}
\mathcal{P}=\left(\begin{matrix}
0.8 & 0.1 & 0.1 \\
0.2 & 0.6 & 0.2 \\
0.2 & 0.2 & 0.6 \\
\end{matrix}\right).
\end{equation*}
For both models,  $\T$ is a Galton-Watson tree with offspring distribution $1+\Rm{Binomial}(2,1/2)$. Results for this experiment are displayed in Figure \ref{fig_gls_extend}.

\begin{figure}[t]
	\centering
	\includegraphics[height=7cm]{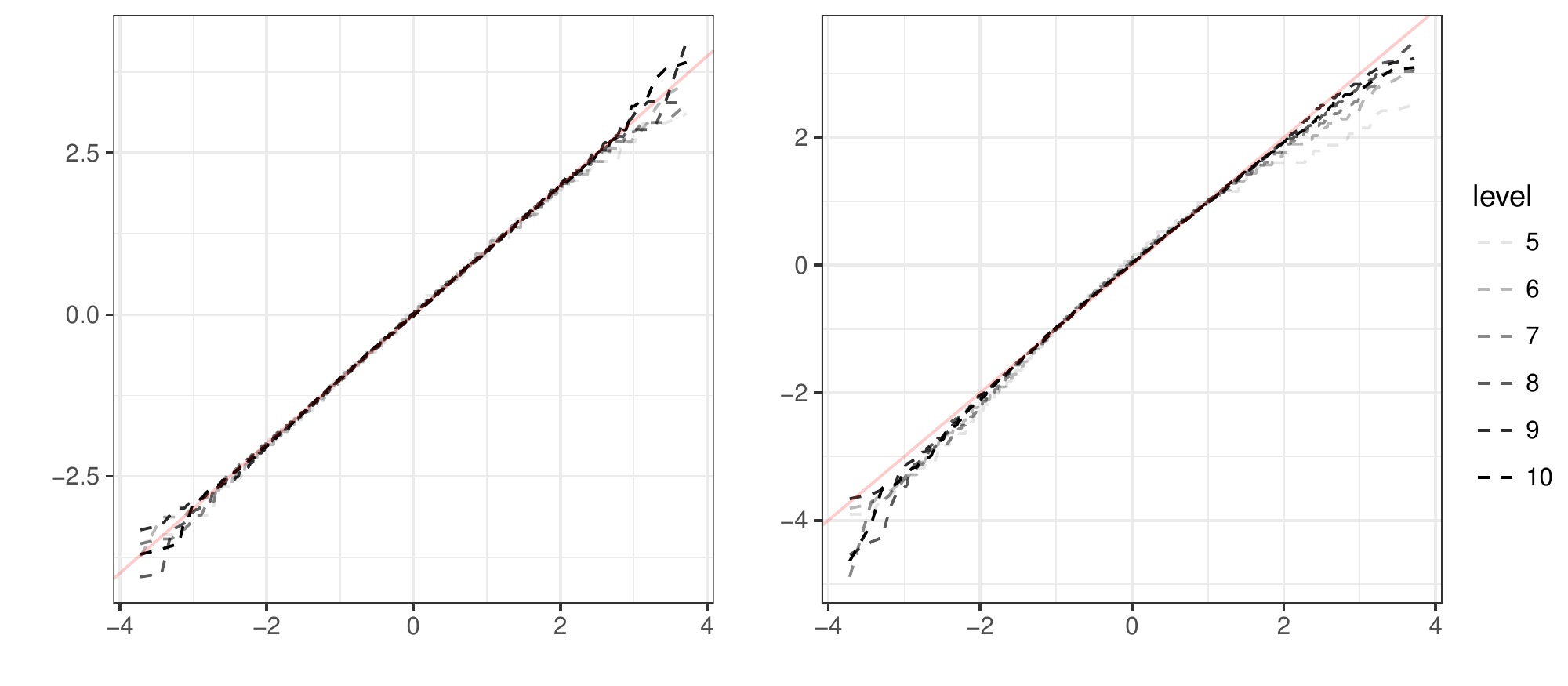}
	\caption{Q-Q plot of $\mu_{t,GLS}$ for the Blockmodels with $2$ blocks, with $m>\lambda_2^{-2}$ (left panel) and $m<\lambda_2^{-2}$ (right panel). $\T$ is a $2$-tree. For each scenario, the Q-Q plot is created over $5000$ replicates. The six dashed Q-Q lines with different colors correspond to $\T$ with 5, 6, 7, 8, 9 or 10 levels. The red solid line is $y=x$. .} 
	\label{fig_gls_2tree}
\end{figure}

\begin{figure}[t]
	\centering
	\includegraphics[height=7cm]{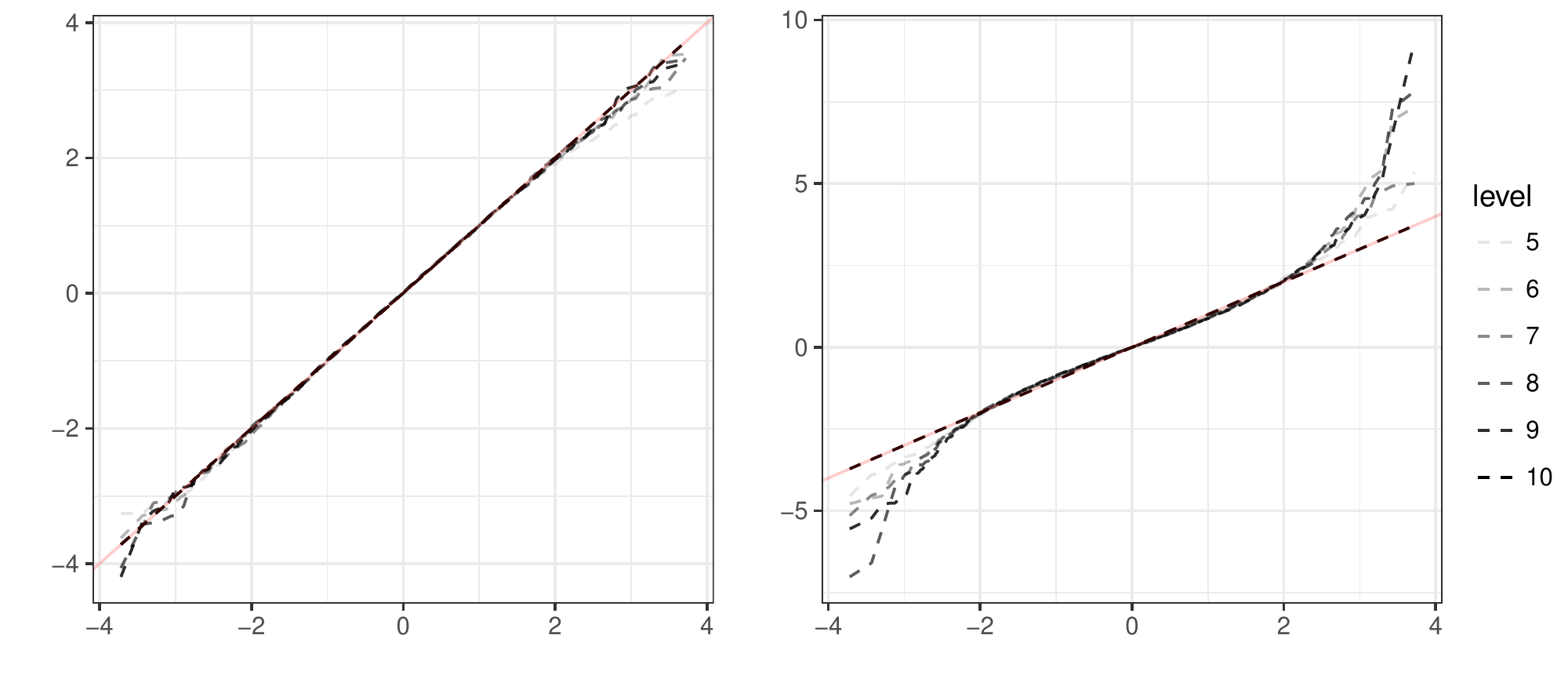}
	\caption{Q-Q plot of $\hat{\mu}_{t,GLS}$ for the Blockmodels with $2$ blocks (left panel) and $3$ blocks (right panel), where $\T$ is a Galton-Watson tree. For each scenario, the Q-Q plot is created over $5000$ replicates. The six dashed Q-Q lines with different colors correspond to $\T$ with 5, 6, 7, 8, 9 or 10 levels. The red solid line is $y=x$. }
	\label{fig_gls_extend}
\end{figure}

\section{Analysis of Adolescent Health Data}\label{section_analysis}
\begin{figure}[t]
	\centering
	\includegraphics[height=14cm]{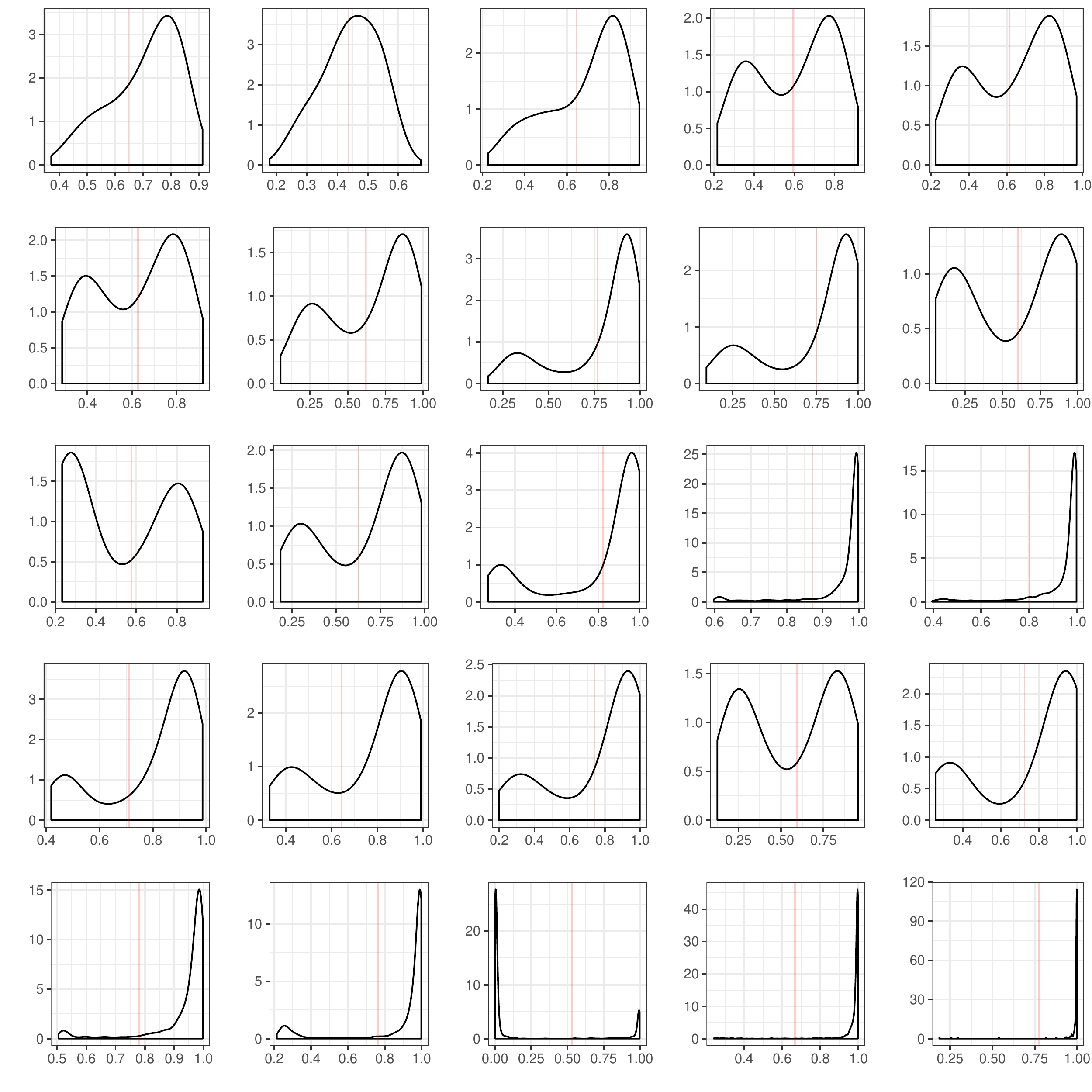}
	\caption{Simulation results based on the Add Health Study described in Section \ref{section_analysis}. The figures display kernel density estimates of the sample average. The 25 subplots correspond to the Comm 17, 75, 42, 15, 28, 39, 40, 41, 50, 34, 45, 48, 36, 43, 61, 54, 59, 73, 44, 68, 60, 58, 84, 57, 49 networks. The red solid line is $x=\mu_{true}$. This figure suggests that VH estimator has multiple modes.}
	\label{fig_adolescent_vh_density}
\end{figure}

\begin{figure}[t]
	\centering
	\includegraphics[height=14cm]{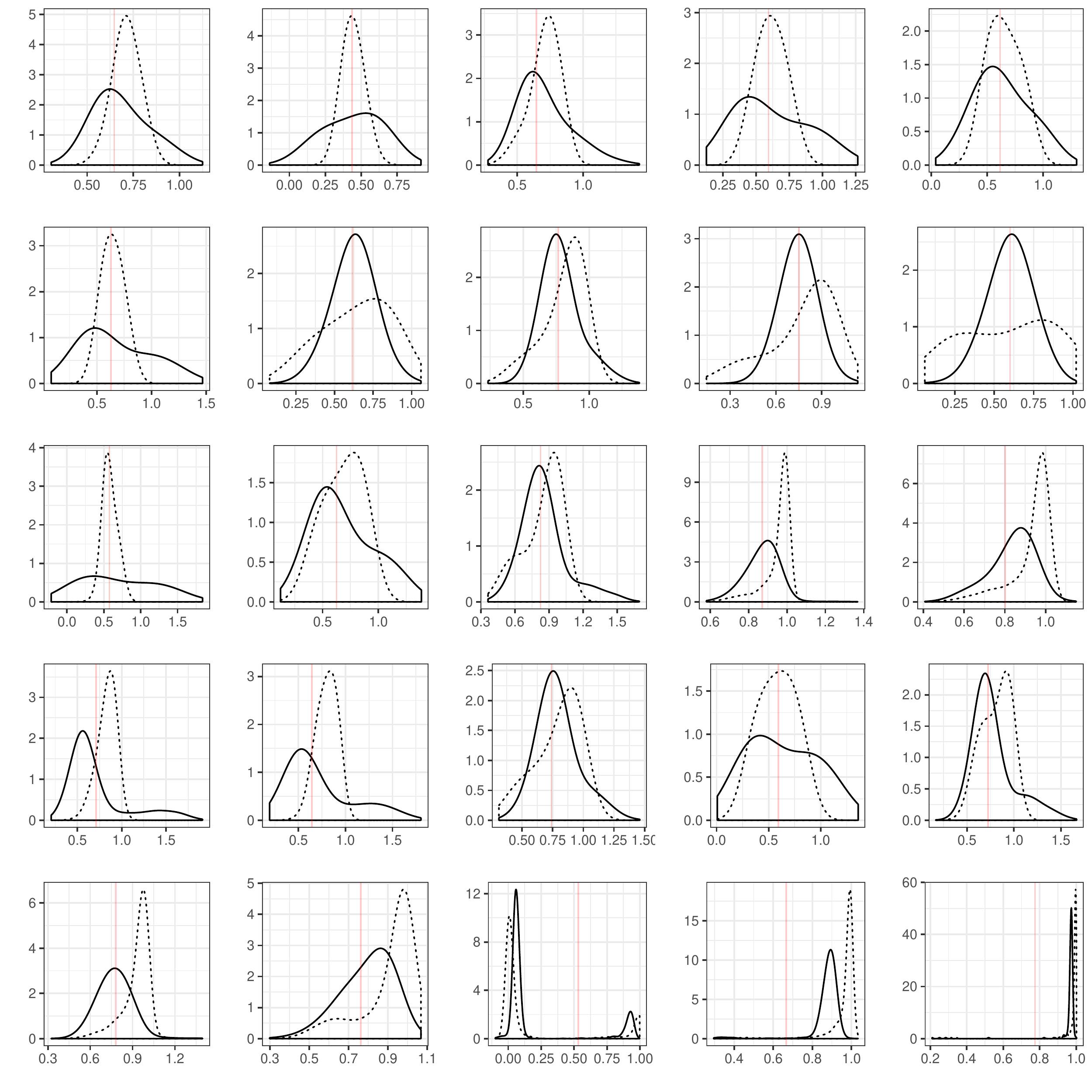}
	\caption{Simulation results based on the Add Health Study described in Section \ref{section_analysis}. The figures display kernel density estimates of the GLS estimator (solid line) and the SBM-fGLS estimator (dashed line). The 25 subplots correspond to the Comm 17, 75, 42, 15, 28, 39, 40, 41, 50, 34, 45, 48, 36, 43, 61, 54, 59, 73, 44, 68, 60, 58, 84, 57, 49 networks. The red solid line is $x=\mu_{true}$. This figure shows that when the bottleneck of the network is not too strong, both estimators have only one mode.}
	\label{fig_adolescent_gls_density}
\end{figure}

\begin{figure}[t]
	\centering
	\includegraphics[height=14cm]{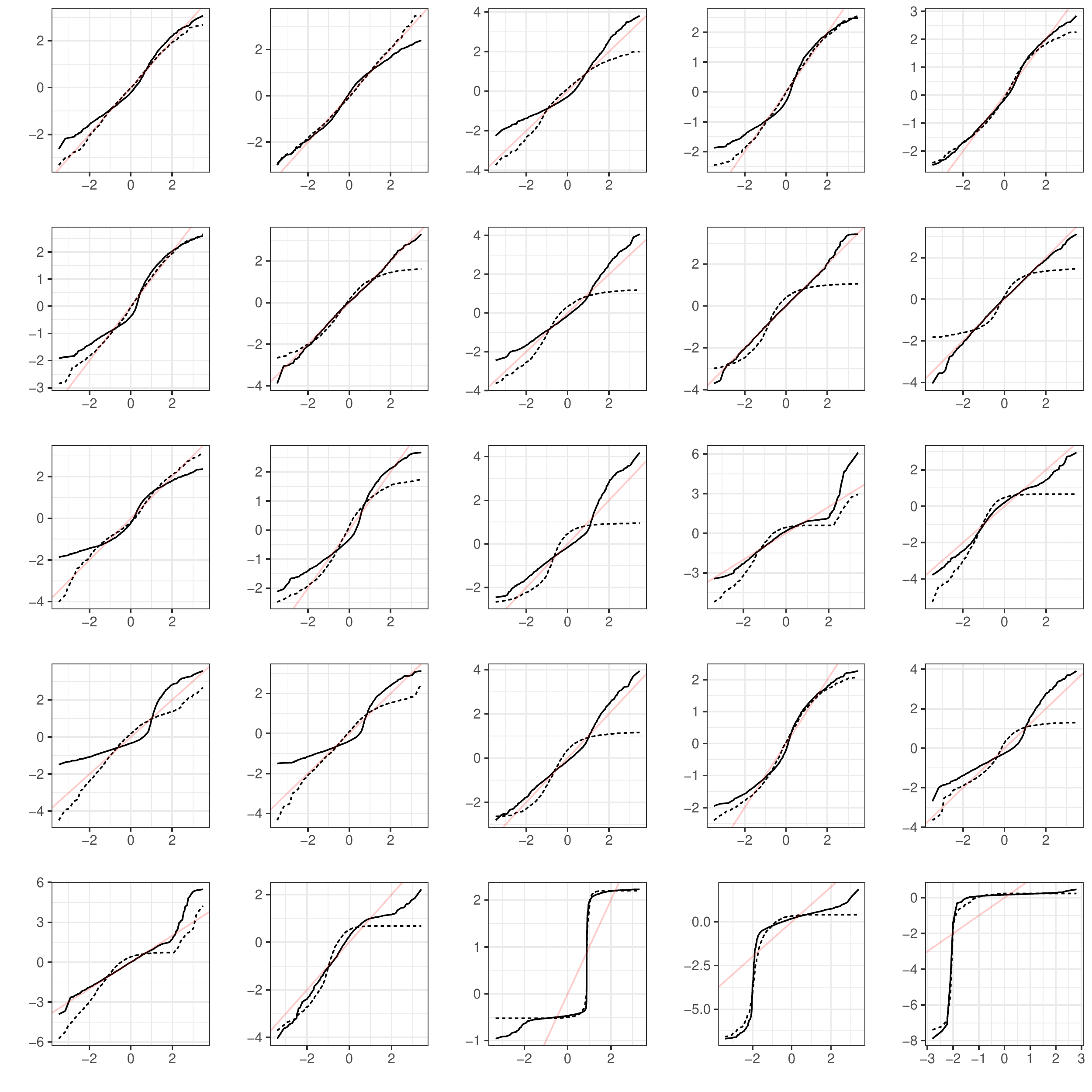}
	\caption{Simulation results based on the Add Health Study described in Section \ref{section_analysis}. The figures display Q-Q plots of the GLS estimator (solid line) and the SBM-fGLS estimator (dashed line). The 25 subplots correspond to the Comm 17, 75, 42, 15, 28, 39, 40, 41, 50, 34, 45, 48, 36, 43, 61, 54, 59, 73, 44, 68, 60, 58, 84, 57, 49 networks. This figure illustrates that when the bottleneck of the network is not too strong, both estimators appear approximately normal (even under without replacement sampling).}
	\label{fig_adolescent_qq}
\end{figure}

In this section, we consider numerical experiments where the RDS samples are simulated without replacement from empirically derived social networks. Specifically, we use social networks collected in the National Longitudinal Study of Adolescent Health (Add Health). In the 1994-95 school year, the Add Health study collected a nationally representative sample of adolescents in grades seven through twelve. The sample covers 84 pairs of middle and high schools in which students nominated up to five male and five female friends in their middle or high school network \citep{harris2011national}. 

In this analysis, we consider 25 networks with at least $1000$ nodes. All contacts are symmetrized and all graphs are restricted to the largest connected component. The RDS sampling process is initialized from a seed node which is selected with probability proportional to node degree (i.e. the stationary distribution). Then, each participant recruits $\xi\sim 1+\Rm{Binomial}(2,1/2)$ participants uniformly at random from their contacts whom have not yet been recruited. If the participant has fewer than $\xi$ contacts eligible to recruit, then the participant recruits all of their eligible contacts. The RDS process stops when there are $500$ participants. If the process terminates before collecting $500$ participants, then the process is restarted. For each network, we collect $500$ different RDS samples. We generate $2000$ such simulated data sets.

We use school-status as the binary node feature and focus on estimating the proportion of the population in high school. We construct a sample average, a GLS estimator and a SBM-fGLS estimator for the proportion of students in high school. The GLS estimator requires an estimate of the covariance matrix  $\Sigma$, which can be calculated from the Markov transition matrix of the network (typically not available in practice) and equation (6) in \cite{rohe2019critical}. The SBM-fGLS estimator proposed in \cite{roch2018generalized} estimates $\Sigma$ using the RDS samples. 

Consider a measure of the network bottleneck. Let $\bA\in\R^{N\times N}$ denote the adjacency matrix of the network. Define the diagonal matrix  $\bD\in\R^{N\times N}$ and the matrix  $\bL\in\R^{N\times N}$ so that
\begin{equation*}
D_{ii}=\sum_{k=1}^{N} W_{ik}, \quad \bL=\bD^{-1/2} \bA \bD^{-1/2}.
\end{equation*}
Then $\tilde{\lambda}$ is defined as
\begin{equation}\label{equ:bottle}
\tilde{\lambda}=\tilde{\by}^\top \bL \tilde{\by}, 
\end{equation}
where $\tilde{\by}$ is the standardized form of the node feature $\by$, so that $\sum_{i=1}^{N} \tilde{y}_i=0$ and $\twonorm{\tilde{\by}}=1$. $\tilde{\lambda}$ provides a measure of the network bottleneck; as long as the second eigenvalue $\lambda_2$ is not too close to $1$, then this quantity will not be close to $1$. Table \ref{table_lambda} displays the $\tilde{\lambda}$ of the $25$ networks. 

\begin{table}[htpb]
	\centering
	\caption{Network characteristics for the  25 networks in the Add Health study used in the numerical experiments in Section \ref{section_analysis}. ID gives the network ID (school ID) from the study listed in increasing order by $\tilde{\lambda}$,  a measure of the strength of bottleneck in the network, see \eqref{equ:bottle}.}
	\label{table_lambda}
	\begin{tabular}{|cc|cc|cc|cc|cc|}
		ID & $\tilde{\lambda}$ & ID & $\tilde{\lambda}$ & ID & $\tilde{\lambda}$ & ID & $\tilde{\lambda}$ &  ID & $\tilde{\lambda}$ \\ \hline
		17   & 0.739             & 39   & 0.842             & 45   & 0.869             & 54   & 0.879             & 60   & 0.911             \\
		75   & 0.744             & 40   & 0.844             & 48   & 0.869             & 59   & 0.881             & 58   & 0.917             \\
		42   & 0.771             & 41   & 0.847             & 36   & 0.874             & 73   & 0.886             & 84   & 0.923             \\
		15   & 0.818             & 50   & 0.867             & 43   & 0.874             & 44   & 0.889             & 57   & 0.925             \\
		28   & 0.839             & 34   & 0.868             & 61   & 0.878             & 68   & 0.897             & 49   & 0.944             \\ 
	\end{tabular}
\end{table}

In Figure \ref{fig_adolescent_vh_density}, the 25 subplots show the kernel density estimation of VH estimator corresponding to the 25 networks. In Figure \ref{fig_adolescent_gls_density} and \ref{fig_adolescent_qq}, the 25 subplots show the kernel density estimation and quantile-quantile plots of GLS and SBM-fGLS estimator with VH adjustment corresponding to the 25 networks. We plot these results over $2000$ replicates. The 25 subplots are in order of descending $\tilde{\lambda}$. It is clear that the VH estimator has two modes, so these networks are all beyond the critical threshold. Except for networks with extremely strong bottleneck (i.e. with large $\tilde{\lambda}$), the GLS estimators with VH adjustment are approximately normally distributed. The distribution of SBM-fGLS estimator with VH adjustment are not enough close to the normal distribution for some networks, which means that our results for the GLS estimator might not always hold for the SBM-fGLS estimator. It is possible for the GLS estimator to exceed one. In practice, one would provide a modified estimate capped at one.

\section{Proof outlines for the main results}
This section outlines the proofs for Theorems \ref{thm_mut} and~\ref{thm_gls}. Well-established theory for multi-type branching processes and martingale limit theorems play an important role. For each proof, the main idea is to extract the underlying martingale structure for the estimator; it is this structure that determines the asymptotic behavior. The proofs of Corollary~\ref{cor_decay}, \ref{cor_VH} and \ref{cor_glsVH} are relegated to Appendices~\ref{S:app-sample}
and~\ref{S:app-GLS}.

\subsection{Analysis of the sample average}
Denote $\bZ_{t,j}$ as the number of $j\in G$ in the $t$-th generation and define $\bZ_t=(Z_{t,1},\ldots,Z_{t,N})$. When $\T$ is an $m$-tree and $X_0$ is initialized from $i\in G$,
$$\bZ_t^{(i)}=(Z_{t,1}^{(i)},\ldots,Z_{t,N}^{(i)})$$
is a multitype Galton-Watson process \citep{harris2002theory,athreya2004branching}. The next lemma can be derived from a standard result in the literature of multitype Galton-Watson processes.
\begin{lemma}
	\label{lem:martingale}
	Assume the conditions of Theorem \ref{thm_mut}. For any $j=1,\ldots,N$, $$Y_{t,j}^{(i)}=(m\lambda_j)^{-t} \langle \bZ_t^{(i)}, \bbf_j\rangle$$ is a real-valued martingale adapted to $\F_t=\sigma\{\bZ_l^{(i)}:1\leq l\leq t\}$.
\end{lemma}
\begin{proof}
	See Appendix \ref{subsec-lem-martingale}.
\end{proof}

Let $W_t$ denote the summation of the $t$-th generation RDS samples,
\[W_t = \sum_{\sigma\in\T:\abs{\sigma}=t} y(X_\sigma),\]
and let $S_t = \sum_{j = 0}^t W_j$ denote the summation up to generation $t$.  Recall that $n_t$ is the number of nodes in $\T$ between the root $0$ and generation $t$ (inclusive), i.e. $n_t = |\{\sigma\in\T,\abs{\sigma}\le t\}|$. Thus the sample average up to generation $t$ is $\hat{\mu}_t = S_t / n_t$. Superscripts on $\bZ, S$ and  $W$ will denote how $X_0$ is initialized if necessary.

Recall from \eqref{eq_eigendecompose} and $\bbf_1=\bm{1}$ that
\begin{equation*}
\by=\sum_{j=1}^{N} \langle \by,\bbf_j \rangle_{\bpi} \bbf_j=\EE_{\bpi}(y)\bm{1}+\sum_{j=2}^{N} \langle \by,\bbf_j \rangle_{\bpi} \bbf_j.
\end{equation*} 
As a result, one obtains the following decomposition of $W_t^{(i)}$
\begin{equation}
\label{eq:Wt-decomposition}
W_t^{(i)}=\sum_{\sigma\in\T,\abs{\sigma}=t} y(X_\sigma^{(i)})=\by^\top \bZ_t^{(i)}=m^t\EE_{\bpi} (y)+\sum_{j=2}^{N} \langle \by,\bbf_j \rangle_{\bpi} (m\lambda_j)^t Y_{t,j}^{(i)}.
\end{equation}
The last step utilizes the simple fact that $\bm{1}^\top\bZ_t^{(i)}=m^t$. This motivates us to study the limit distribution of $Y_{t,j}^{(i)}$. 

\begin{lemma}
	\label{lem:martingale-convergence}
	Assume the conditions of Theorem \ref{thm_mut}. Then there exists a random variable $Y_2^{(i)}$ such that $$Y_{t,2}^{(i)} \to Y_2^{(i)}$$ almost surely and in $L^2$.
	For $j\geq 3$, 
	$$(\lambda_2^{-1}\lambda_j)^t\,Y_{t,j}^{(i)} \to 0$$ 
	almost surely and in $L^2$.
\end{lemma}
\begin{proof}
	See Appendix \ref{subsec-lem-martingale-converence}.
\end{proof}

Lemma \ref{lem:martingale-convergence} informally reveals that, under proper scaling, the asymptotic distributional characterization of $W_t^{(i)}$ is determined by $Y_{t,2}^{(i)}$. The next lemma derives the first and second moments of $Y_2^{(i)}$.
\begin{lemma}
	\label{lem:nondegenerate-Y2} 
	Assume the conditions of Theorem \ref{thm_mut}. Then $\EE Y_2^{(i)}=f_2(i)$, and $\mathsf{Var}(Y_2^{(i)})>0$ for any $i=1,\ldots,N$ if we further assume that $\langle \by, \bbf_2 \rangle_{\bpi}\neq 0$ holds.
\end{lemma}
\begin{proof}
	See Appendix \ref{subsec-lem-Y2}.
\end{proof}

\subsection{Proof of Theorem \ref{thm_mut}}
	Apply \eqref{eq:Wt-decomposition} and Lemma \ref{lem:martingale-convergence} collectively to obtain
	\begin{equation}
	\label{eq_Wt}
	\frac{W_t^{(i)}-m^t\EE_{\bpi} (y)}{(m\lambda_2)^t}=\sum_{j=2}^{N} \langle \by,\bbf_j \rangle_{\bpi} (\lambda_2^{-1}\lambda_j)^t Y_{t,j}^{(i)} \rightarrow \langle \by,\bbf_2 \rangle_{\bpi} Y_{2}^{(i)}
	\end{equation}
	almost surely and in $L^2$. Recalling that $S_t=\sum_{j=0}^t W_j$, $\hat{\mu}_t=S_t/n_t$, 
	and the number of samples between $0$ and generation $t$ is
	$n_t=\sum_{l=0}^{t}m^l$, one arrives at
	\begin{equation}
	\label{eq_as_converge}
	\lambda_2^{-t}\left[\hat{\mu}_t^{(i)}-\EE_{\bpi} (y)\right]=\frac{S_t^{(i)}-n_t \EE_{\bpi} (y)}{n_t\lambda_2^t}=\frac{m^t}{n_t} \sum_{l=0}^{t} (m\lambda_2)^{l-t}\frac{W_l^{(i)}-m^l \EE_{\bpi} (y)}{(m\lambda_2)^l}.
	\end{equation}
	Since $\lim_{t\to\infty}m^t/n_t=(m-1)/m$, from \eqref{eq_Wt} 
	\begin{equation}
	\label{eq:mu_converge}
	\lambda_2^{-t}\left[\hat{\mu}_t^{(i)}-\EE_{\bpi} (y)\right] \to \frac{m-1}{m}\sum_{r=0}^{\infty} (m\lambda_2)^{-r}\langle \by,\bbf_2 \rangle_{\bpi} Y_{2}^{(i)}=\frac{(m-1)\lambda_2}{m\lambda_2-1}\langle \by,\bbf_2 \rangle_{\bpi} Y_{2}^{(i)}\triangleq X^{(i)}
	\end{equation}
	almost surely as $t\to\infty$. 
	
	To prove $L^2$ convergence, recall that if a sequence of random variables $X_n \to X$ in probability, and $\Vert X_n \Vert_{L^2} \to \Vert X \Vert_{L^2}$, then $X_n \to X$ in $L^2$. Observe, similarly to the above limits, that
	\begin{align*}
	\left\Vert \lambda_2^{-t}(\hat{\mu}_t^{(i)}-\EE_{\bpi} (y)) \right\Vert_{L^2}^2&=\frac{m^{2t}}{n_t^2}\EE \left[\left(\sum_{l=0}^{t} \frac{W_l^{(i)}-m^l \EE_{\bpi} (y)}{(m\lambda_2)^t}\right)^2\right]\\
	&=\frac{m^{2t}}{n_t^2}\sum_{k=1}^{t}\sum_{l=1}^{t} \EE\left\{(m\lambda_2)^{-2t}\left[W_k^{(i)}-m^k \EE_{\bpi} (y)\right] \left[W_l^{(i)}-m^l \EE_{\bpi} (y)\right]\right\} \\
	&\xrightarrow{t\to\infty}\lim_{r\to\infty}
	\left(\frac{m-1}{m}\right)^2\EE\left[\left(\langle \by,\bbf_2 \rangle_{\bpi} Y_{2}^{(i)}\right)^2\right]\sum_{k=1}^{r}\sum_{l=1}^{r} (m\lambda_2)^{k+l-2r} \\
	&=\EE\left[\left(\frac{(m-1)\lambda_2}{m\lambda_2-1} \langle \by,\bbf_2 \rangle_{\bpi} Y_{2}^{(i)}\right)^2\right].
	\end{align*}
	So the convergence in \eqref{eq:mu_converge} is also in $L_2$.
	
	Finally, in view of Lemma \ref{lem:nondegenerate-Y2} and the definition of $X^{(i)}$ in \eqref{eq:mu_converge}, it is straightforward to check that \eqref{eq_different_expectation} holds. Moreover, if $\langle \by,\bbf_2\rangle_{\bpi}\neq0$, then $\mathsf{Var}(X^{(i)})>0$ for any $i=1,\ldots,N$.

\begin{remark}
	\label{remark:generalize_eig}
	If condition \eqref{eq_eigenvalue} in Theorem \ref{thm_mut} is weakened to
	\begin{equation*}
	1=\lambda_1>\lambda_2=\cdots=\lambda_k>\abs{\lambda_{k+1}}\geq\cdots\geq\abs{\lambda_N},
	\end{equation*}
	then \eqref{eq_mu_converge} becomes
	\begin{equation*}
	\lambda_2^{-t} (\hat{\mu}_t^{(i)}-\EE_{\bpi} (y)) \rightarrow \frac{(m-1)\lambda_2}{m\lambda_2-1} \sum_{j=2}^{k} \langle \by,\bbf_j \rangle_{\bpi}Y_{j}^{(i)}.
	\end{equation*}
\end{remark}

\subsection{Analysis of the GLS estimator}
In previous sections, the subscript of the estimators is $t$ or $l$, which denotes the generation. This section requires us to study each node in a generation. Accordingly we order the nodes of the $m$-tree $\T$ by scanning each level from the root down. For example, for a $2$-tree, the root node is $1$, its offsprings are $2$ and $3$, the offsprings of $2$ are $4$ and $5$, the offsprings of $3$ are $6$ and $7$, etc. In a change of notation from the previous sections, when the subscript is $n$, $\hat{\mu}^{(\cdot)}_{n}$ now denotes the sample mean up to node $n$, i.e.
$\hat{\mu}_{n}=n^{-1}\sum_{k=1}^{n} y(X_k)$. 

Assume the Markov transition matrix between blocks is
$$\cP=\Big(\begin{matrix}
p & 1-p \\ 1-q & q
\end{matrix}\Big).$$
The second eigenvalue is $\lambda_2=p+q-1$ and the stationary
distribution is $\bpi = (\frac{1-q}{1-\lambda_2},\frac{1-p}{1-\lambda_2})$. For $k\geq 1$, define
\begin{equation}
\label{eq:Mn-martingale}
M_n=\sum_{k=1}^{n} \left[y(X_k^{(\bnu)})-\lambda_2 y(X_{p(k)}^{(\bnu)})\right]-n(1-\lambda_2)\EE_{\bpi} (y),
\end{equation}
where $p(k)$ is the parent node of $k$ in the ordering
defined above.
In view of \eqref{eq_gls_expression}, the relation between $M_n$ and $\hat{\mu}_{\mathsf{GLS,t}}$ is 
\begin{equation}
\label{eq:mtg-gls-relationship}
\sqrt{n_t}\left[\hat{\mu}_{\mathsf{GLS},t}^{(\bnu)}-\EE_{\bpi} (y)\right]=\frac{M_{n_t}}{\sqrt{n_t}(1-\lambda_2)}+O_{\PP}\left(\frac{1}{\sqrt{n_t}}\right).
\end{equation}
Thus it suffices to study the asymptotic behavior of $M_n$. One can check that $M_n$ is a martingale adapted to the filtration
$\F_n=\sigma(X_k^{(\nu)}: k\leq n)$:
\begin{equation*}
\begin{aligned}
\EE\left(M_n-M_{n-1}\mid \cF_{n-1}\right)&=\begin{cases}
p y_1+(1-p) y_2-\lambda_2 y_1-(1-\lambda_2)\EE_{\bpi} (y) \ &\text{if}  \ X_{p(n)}^{(\bnu)}=1\\
(1-q) y_1+q y_2-\lambda_2 y_2-(1-\lambda_2)\EE_{\bpi} (y) \ &\text{if} \ X_{p(n)}^{(\bnu)}=2
\end{cases} =0,
\end{aligned}
\end{equation*}
both of which are $0$, as can be seen from
$(1-\lambda_2)\EE_{\bpi} (y)=(1-q) y_1+(1-p) y_2$ and the expression for $\lambda_2$ above.
It is necessary to introduce a martingale central limit theorem (see e.g. \cite[Fifth Edition, Theorem 8.2.8]{durrett2019probability}).
\begin{theorem}[Martingale CLT]
	\label{thm_martingale_CLT}
	Let a martingale $M_n$ satisfy $\EE(M_n)=0$, and
	\begin{enumerate}
		\item $\frac{1}{n} \sum_{k=1}^{n} \EE((M_k-M_{k-1})^2\mid M_1,\ldots,M_{k-1})\rightarrow \sigma^2>0$ in probability as $n\to \infty$, and
		\item for every $\epsilon>0$, $\frac{1}{n} \sum_{k=1}^{n} \EE((M_k-M_{k-1})^2;  \abs{M_k-M_{k-1}}>\epsilon\sqrt{n})\rightarrow 0$ as $n\to\infty$,
	\end{enumerate}
	then $M_n/\sqrt{n}\to \cN(0,\sigma^2)$ in distribution as $n\to\infty$.
\end{theorem}

It then boils down to showing that $M_n$ defined in \eqref{eq:Mn-martingale} satisfies the conditions in the above theorem. The detailed proof is provided in the next section. We will need a technical lemma which states that, although the limit distribution of the sample average $\hat{\mu}_n^{(i)}$ differs in the high and low variance regimes under appropriate scalings, $\hat{\mu}_n^{(i)}$ itself always converges to $\EE_{\bpi}(y)$ in $L^2$. 

\begin{lemma}
	\label{lem:trivial'}
	Assume the conditions of Theorem \ref{thm_gls}. Then for any initial distribution $\bnu$ of $X_0$, $\hat{\mu}_n^{(\bnu)} \to \EE_{\bpi}(y)$ in $L^2$.
\end{lemma}
\begin{proof}
	See Appendix \ref{subsec-lem-trivial'}
\end{proof}

\subsection{Proof of Theorem \ref{thm_gls}}
\label{subsec-theorem-gls}	
Without loss of generality, assume $\EE_{\bpi}(y)=0$ and $\mathsf{Var}_{\bpi}(y)=1$, which can be equivalently viewed as applying the same linear transformation to each entry of $\by$ as well as $\hat{\mu}_{\mathsf{GLS},t}$.

We begin by showing that $M_n$ defined in \eqref{eq:Mn-martingale} satisfies the first condition in Theorem \ref{thm_martingale_CLT}.
By invoking the martingale property $\EE(M_k-M_{k-1} \mid \F_{k-1})=0$, one obtains
	\begin{equation*}
	\EE\left[(M_k-M_{k-1})^2\mid \F_{k-1}\right]=\mathsf{Var}\left(M_k-M_{k-1}\mid \F_{k-1}\right)=\mathsf{Var}\big(y(X_k^{(\bnu)})\mid X_{p(k)}^{(\bnu)}\big).
	\end{equation*}
	Notice that $\EE_{\bpi}(y)=0$ and $\mathsf{Var}_{\bpi} (y)=1$ imply $(1-q)y_1^2+(1-p)y_2^2=2-p-q$, which yields
	\begin{equation*}
	\begin{aligned}
	\mathsf{Var}\big(y(X_k^{(\bnu)})\mid X_{p(k)}^{(\bnu)}\big)&=\begin{cases}
	py_1^2+(1-p)y_2^2-\lambda_2^2 y_1^2 &\text{if}  \ X_{p(n)}^{(\bnu)}=1\\
	(1-q)y_1^2+qy_2^2-\lambda_2^2 y_2^2 &\text{if} \ X_{p(n)}^{(\bnu)}=2
	\end{cases}\\
	&=(1-\lambda_2)(1+\lambda_2 y(X_{p(k)}^{(\bnu)})^2).
	\end{aligned}
	\end{equation*}
	Thus
	$\EE[(M_k-M_{k-1})^2\mid \F_{k-1}]=(1-\lambda_2)(1+\lambda_2 y(X_{p(k)}^{(\bnu)})^2)$.
	For notational simplicity, denote
	$$V_n=\frac{1}{n} \sum_{k=1}^{n} \EE\left[(M_k-M_{k-1})^2\mid \F_{k-1}\right]=\frac{1-\lambda_2}{n} \sum_{k=1}^{n} (1+\lambda_2 y(X_{p(k)}^{(\bnu)})^2).$$
	When $\T$ is an $m$-tree, each node from level $0$ to $t-1$ is counted $m$ times as a parent. Define a new node feature $\by^\prime=(y(1)^2,y(2)^2)^\top$. Let $\hat{\gamma}_{n}^{(\bnu)}$ be the sample average of $y^\prime(X_\sigma^{(\bnu)})$'s up to node $n$. By Lemma \ref{lem:trivial'} applied to $\hat{\gamma}_{n}^{(\bnu)}$,
	\begin{equation*}
	\begin{aligned}
	V_n&=1-\lambda_2+\lambda_2(1-\lambda_2)\frac{m\sum_{k=1}^{p(n)}y'(X_k^{(\nu)})+O(1)}{n}=1-\lambda_2+\lambda_2(1-\lambda_2)\hat{\gamma}_{p(n)}^{(\bnu)}+O(1/n)\\
	&\xrightarrow{L^2} 1-\lambda_2+\lambda_2(1-\lambda_2) \EE_{\bpi} (y')=1-\lambda_2^2
	\end{aligned}
	\end{equation*}
	as $n\to\infty$. Here $O(1)$ in the first line comes from the fact that $y'(X_{p(n)}^{(\bnu)})$ might be counted less than $m$ times, which results in a remainder term bounded by $m\Vert \by^\prime\Vert_\infty$, and the second line uses $\EE_{\bpi}(y^\prime)=\mathsf{Var}_{\bpi}(y)=1$. Since $L^2$ convergence implies convergence in probability, the first condition is verified.
	
	We now move on to the second condition. Notice that $$\abs{M_k-M_{k-1}}=\abs{y(X_k^{(\bnu)})-\lambda_2 y(X_{p(k)}^{(\bnu)})}\leq (1+\lambda_2) \inftynorm{\by}.$$
	So $\PP(\abs{M_k-M_{k-1}}>\epsilon\sqrt{n})=0$ when $n>\epsilon^{-2}(1+\lambda_2)^2 \inftynorm{\by}^2$.
	This gives
	\begin{equation*}
	\frac{1}{n} \sum_{k=1}^{n} \EE((M_k-M_{k-1})^2;  \abs{M_k-M_{k-1}}>\epsilon\sqrt{n})=0
	\end{equation*}
	for sufficiently large $n$. Thus the second condition is verified.
	
	As a result, we obtain from Theorem \ref{thm_martingale_CLT} that $M_n/\sqrt{n} \to \cN(0,1-\lambda_2^2)$ in distribution. Combined with \eqref{eq:mtg-gls-relationship} and Slutsky's theorem, one finally arrives at
	\begin{equation*}
	\sqrt{n_t}\left[\hat{\mu}_{\mathsf{GLS},t}^{(\bnu)}-\EE_{\bpi}(y)\right] \xrightarrow{d}\cN\left(0,\frac{1+\lambda_2}{1-\lambda_2}\,\mathsf{Var}_{\bpi}(y)\right).
	\end{equation*}	

\section{Discussion}\label{sec:disc}
We prove the existence of a limit distribution for the IPW estimator under the Markov model of respondent-driven sampling and show that this limit distribution depends on the seed node---thus the limit distribution is a non-trivial mixture distribution when the seed is randomized. This result also shows that the ``seed bias'' of IPW is non-negligible. We also establish the asymptotic normality of the GLS estimator under certain conditions and show that this limiting normal does not depend on the seed node. This implies that the ``seed bias'' of GLS is negligible. Both results allow for the VH adjustment. Our empirical study on social networks as well as on simulated data illustrate that these theoretical results appear to hold beyond the technical conditions given in the theorems. 
\section*{Acknowledgements}
Yan is partially supported by the elite undergraduate training program of School of Mathematical Sciences in Peking University. Roch is supported by NSF grants DMS-1614242 CCF-1740707 (TRIPODS) and DMS-1916378, and a Simons Fellowship. Rohe is supported by NSF grant DMS-1612456 and ARO grant W911NF-15-1-0423.

\appendix

\section{Proof of Equation \eqref{eq:blockmodel}}
\label{section_proof_eq2}
In this section, we use $Z$ and $z$ in place of $B$ and $b$.
First we use mathematical induction to show that, for every $n\in\Z^+$, the following statement $P(n)$ holds: 

For any given referral tree $\T$ with $n$ vertices $\{\sigma_1,\ldots,\sigma_n\}$, for any initial distribution $\bnu$ of $X_0$ and $z_1,\ldots,z_n\in\{1,\ldots,k\}$, the following holds
\begin{equation}
\label{eq:math_induction}
\PP(Z_{\sigma_1}^{(\bmu)}=z_1,\ldots,Z_{\sigma_n}^{(\bmu)}=z_n)=\PP(z(X_{\sigma_1}^{(\bnu)})=z_1,\ldots,z(X_{\sigma_n}^{(\bnu)})=z_n)
\end{equation}
with $\mu_j=\sum_{i\in G:z(i)=j} \nu_i$ for $j=1,\ldots,k$. 

Base case: We prove that $P(1)$ holds. Since $\T$ only contains the seed vertex $0$, it suffices to show that $\PP(Z_0^{(\bmu)}=z_0)=\PP(z(X_0^{(\bnu)})=z_0)$ for any $z_0\in\{1,\ldots,k\}$. However,
\begin{equation*}
\PP(z(X_0^{(\bnu)})=z_0)=\sum_{i\in G: z(i)=z_0} \nu_i=\mu_{z_1}=\PP(Z_0^{(\bmu)}=z_0).
\end{equation*}
So $P(0)$ is true.

Inductive step: We prove that if $P(n-1)$ holds for some unspecified value of $n\geq 2$, then $P(n)$ also holds. Assume $\sigma_{n}$ is a leaf node (i.e. $\sigma_{n}$ has no descendant) and $\sigma_{n-1}$ is the parent of $\sigma_{n}$. Then $\T\setminus\{\sigma_{n}\}$ is a referral tree with $n-1$ vertex. By the Markov property,
\begin{equation*}
\begin{aligned}
&\PP(z(X_{\sigma_n}^{(\bnu)})=z_n\mid z(X_{\sigma_1}^{(\bnu)})=z_1,\ldots,z(X_{\sigma_{n-1}}^{(\bnu)})=z_{n-1})\\
= &\frac{\sum_{i\in G:z(i)=z_{n-1}} \PP(z(X_{\sigma_n}^{(\bnu)})=z_n\mid X_{\sigma_{n-1}}^{(\bnu)}=i)\PP(X_{\sigma_{n-1}}^{(\bnu)}=i\mid z(X_{\sigma_1}^{(\bnu)})=z_1,\ldots,z(X_{\sigma_{n-2}}^{(\bnu)})=z_{n-2})}{\PP(z(X_{\sigma_{n-1}}^{(\bnu)})
=z_{n-1}\mid z(X_{\sigma_1}^{(\bnu)})=z_1,\ldots,z(X_{\sigma_{n-2}}^{(\bnu)})=z_{n-2})}\\
=&\frac{\sum_{i\in G:z(i)=z_{n-1}}\mathcal{P}_{z(i)z_n}\PP(X_{\sigma_{n-1}}^{(\bnu)}=i\mid z(X_{\sigma_1}^{(\bnu)})=z_1,\ldots,z(X_{\sigma_{n-2}}^{(\bnu)})=z_{n-2})}{\PP(z(X_{\sigma_{n-1}}^{(\bnu)})
	=z_{n-1}\mid z(X_{\sigma_1}^{(\bnu)})=z_1,\ldots,z(X_{\sigma_{n-2}}^{(\bnu)})=z_{n-2})}
=\mathcal{P}_{z_{n-1}z_n}\\
=&\PP(Z_{\sigma_n}^{(\bmu)}=z_n\mid Z_{\sigma_1}^{(\bmu)}=z_1,\ldots,Z_{\sigma_{n-1}}^{(\bmu)}=z_{n-1}).
\end{aligned}
\end{equation*}
Additionally, the induction hypothesis that $P(n-1)$ holds gives
\begin{equation*}
\PP(Z_{\sigma_1}^{(\bmu)}=z_1,\ldots,Z_{\sigma_{n-1}}^{(\bmu)}=z_{n-1})=\PP(z(X_{\sigma_1}^{(\bnu)})=z_1,\ldots,z(X_{\sigma_{n-1}}^{(\bnu)})=z_{n-1}).
\end{equation*}
The above two equations give \eqref{eq:math_induction}, thereby showing $P(n)$ is true.

Since both the base case and the inductive step have been performed, by mathematical induction the statement $P(n)$ holds for all $n\in\Z^+$. 

Finally we prove \eqref{eq:blockmodel} based on the above result. Assume $\T$ has $n$ vertices. For any $\{\sigma_{i_1}, \ldots, \sigma_{i_s}\}\subset\T$ and $z_{i_1},\ldots,z_{i_s}\in\{1,\ldots,k\}$, let $\{\sigma_{j_1}, \ldots, \sigma_{j_{n-s}}\}=\T\setminus\{\sigma_{i_1}, \ldots, \sigma_{i_s}\}$. Then
\begin{equation*}
\begin{aligned}
\PP(Z_{\sigma_{i_1}}^{(\cdot)}=
z_{i_1},\ldots,Z_{\sigma_{i_s}}^{(\cdot)}=z_{i_s})&=\sum_{z_{j_1}=1}^{k}\cdots\sum_{z_{j_{n-s}}=1}^{k} \PP(Z_{\sigma_1}^{(\cdot)}=z_1,\ldots,Z_{\sigma_n}^{(\cdot)}=z_n)
\\
&=\sum_{z_{j_1}=1}^{k}\cdots\sum_{z_{j_{n-s}}=1}^{k} \PP(z(X_{\sigma_1}^{(\cdot)})=z_1,\ldots,z(X_{\sigma_n}^{(\cdot)})=z_n)\\
&=\PP(z(X_{\sigma_{i_1}}^{(\cdot)})=z_{i_1},\ldots,z(X_{\sigma_{i_s}}^{(\cdot)})=z_{i_s}).
\end{aligned}
\end{equation*}

\section{Proofs: sample average}
\label{S:app-sample}

Define the mean matrix $\bM\in\R^{N\times N}$ as
$$\bM=\{\EE Z_{1,j}^{(i)}:i,j=1,\ldots,N\}.$$ 
Let $\bV_i$ denote the variance-covariance matrix of $\bZ_1^{(i)}$, and define $$\bC_t^{(i)}=\{\EE Z_{t,j}^{(i)} Z_{t,k}^{(i)}: j,k=1,\ldots,N\}.$$
All components of $\bM$ and $\bC_t^{(i)}$ are finite. The following lemma is a standard result of multitype Galton-Watson process, see e.g. \cite{harris2002theory} or \cite{athreya2004branching}.

\begin{lemma}
	The expectation of $\bZ_t^{(i)}$ and $\bC_t^{(i)}$ can be calculated from
	\begin{equation}
	\label{E_general}
	\EE (\bZ_t^{(i)})=\bZ_0^{(i)} \bM^t, \quad \Rm{and}
	\end{equation}
	\begin{equation}
	\label{var_iteration}
	\bC_t^{(i)}=(\bM^\top)^t \bC_0^{(i)} \bM^t+\sum_{l=1}^{t} (\bM^\top)^{t-l} \left(\sum_{k=1}^{N}\bV_k\EE Z_{l-1,k}^{(i)}\right) \bM^{t-l}.
	\end{equation}
\end{lemma}

\subsection{Proof of Lemma \ref{lem:martingale}}
\label{subsec-lem-martingale}
For the Markov model, $\bM=m\bP$ so $\bbf_j$ is the eigenvector of $\bM$ corresponding to the eigenvalue $m\lambda_j$. The following lemma comes from the well-established theory of multitype Galton-Watson process.
\begin{lemma}
	\label{lem-branching}
	Let $\bxi$ be a right eigenvector of $\bM$ and $\lambda$ be the corresponding eigenvalue. Then
	$$\lambda^{-t}\langle \bZ_t^{(i)}, \bxi \rangle$$
	is a (complex-valued) martingale adapted to $\F_t=\sigma(\bZ_l^{(i)}:1\leq l \leq t)$.
\end{lemma}
\begin{proof}
	See Theorem 4' on Page 196 of \cite{athreya2004branching}.
\end{proof}

According to Lemma \ref{lem_eigendecomposition}, all $\lambda_j$ and $\bbf_j$ are real. One then applies Lemma \ref{lem-branching} to the Markov model with $\lambda=m\lambda_j$ and $\bxi=\bbf_j$ to complete the proof.

\subsection{Proof of Lemma \ref{lem:martingale-convergence}}
\label{subsec-lem-martingale-converence}
The next theorem is the martingale $L^p$ convergence theorem (see e.g. \cite{durrett2019probability}).

\begin{theorem}
	\label{martingale_convergence_thm}
	If $X_n$ is a martingale with $\sup_n \EE \abs{X_n}^p<\infty$ where $p>1$, then $X_n \to X$ almost surely and in $L^p$.
\end{theorem}

It is essential to derive the variance of $\langle \bZ_t^{(i)}, f_j \rangle$ before applying Theorem \ref{martingale_convergence_thm} to the martingales $Y_{t,j}^{(i)}, j\geq 2$. We conclude the result in the following claim and defer the proof to the end of this section. 

\begin{claim}
	\label{var_thm}
	The variance of $\langle \bZ_t^{(i)}, \bbf_j \rangle$ is
	\begin{equation}
	\label{eq_var}
	\mathsf{Var}(\langle \bZ_t^{(i)}, \bbf_j \rangle)=\begin{cases}
	O((m\lambda_j)^{2t}) & \Rm{if} \ m\lambda_j^2 >1, \\
	O(t(m\lambda_j)^{2t}) & \Rm{if} \ m\lambda_j^2=1, \\
	O(m^t) & \Rm{if} \ m\lambda_j^2 <1. \\
	\end{cases}
	\end{equation}
\end{claim}

We begin with $Y_2^{(i)}$. By Theorem \ref{martingale_convergence_thm}, we only need to show $\sup_t \EE (Y_{t,2}^{(i)}) ^2<\infty$. However,
\begin{equation*}
\EE (Y_{t,2}^{(i)})^2=\mathsf{Var}(Y_{t,2}^{(i)})+(\EE Y_{t,2}^{(i)})^2=(m\lambda_2)^{-2t}\,\mathsf{Var}(\langle \bZ_t^{(i)}, \bbf_2 \rangle)+(Y_{0,2}^{(i)})^2.
\end{equation*}
Since $m>\lambda_2^{-2}$, by Claim \ref{var_thm}, $\mathsf{Var}(\langle \bZ_t^{(i)}, \bbf_2 \rangle)=O((m\lambda_2)^{2t})$. This gives $\sup_t \EE (Y_{t,2}^{(i)}) ^2<\infty$.

We move on to $Y_j^{(i)}$ for $j\geq 3$. By Theorem \ref{var_thm},
\begin{equation*}
\begin{aligned}
\EE [(\lambda_2^{-1}\lambda_j)^tY_{t,j}^{(i)}]^2
&=(\lambda_2^{-1}\lambda_j)^{2t}[\mathsf{Var}(Y_{t,2}^{(i)})+(\EE Y_{t,2}^{(i)})^2]\\
&=(m\lambda_2)^{-2t}\,\mathsf{Var}(\langle \bZ_t^{(i)}, \bbf_j \rangle)+(\lambda_2^{-1}\lambda_j)^{2t} (Y_{0,2}^{(i)})^2. 
\end{aligned}
\end{equation*}
Since $\lambda_2^{-1}\lambda_j<1$, $(\lambda_2^{-1}\lambda_j)^{2t} (Y_{0,2}^{(i)})^2\to 0$. Additionally, for $j\geq 3$,
\begin{equation*}
(m\lambda_2)^{-2t}\,\mathsf{Var}(\langle \bZ_t^{(i)}, \bbf_j \rangle)=\begin{cases}
O((\lambda_2^{-1}\lambda_j)^{2t}) & \Rm{if} \ m\lambda_j^2 >1 \ \\
O(t(\lambda_2^{-1}\lambda_j)^{2t}) & \Rm{if} \ m\lambda_j^2=1 \ \\
O((m\lambda_2^2)^{-t}) & \Rm{if} \ m\lambda_j^2 <1 \ \\
\end{cases}
\end{equation*}
which converges to $0$ in all cases.
Thus $\EE [(\lambda_2^{-1}\lambda_j)^tY_{t,j}^{(i)}]^2 \to 0$, which leads to $(\lambda_2^{-1}\lambda_j)^tY_{t,j}^{(i)} \to 0$ in $L^2$. 
To prove almost sure convergence, let $\delta=\max\{(\lambda_2^{-1} \lambda_j)^2,m\lambda_2^{-2}\}\in(0,1)$. There exists $C>0$ such that
\begin{equation*}
\EE [(\lambda_2^{-1}\lambda_j)^tY_{t,j}^{(i)}]^2\leq Ct\delta^t
\end{equation*}
always holds. Then $\forall\epsilon>0$, 
\begin{equation*}
\PP(|(\lambda_2^{-1}\lambda_j)^tY_{t,j}^{(i)}|>\epsilon)\leq \epsilon^{-2}\EE [(\lambda_2^{-1}\lambda_j)^tY_{t,j}^{(i)}]^2\leq \epsilon^{-2} Ct\delta^t.
\end{equation*}
So
\begin{equation*}
\sum_{t=1}^{\infty} \PP(|(\lambda_2^{-1}\lambda_j)^tY_{t,j}^{(i)}|>\epsilon)\leq\epsilon^{-2} C \sum_{t=1}^{\infty} t\delta^t<\infty.
\end{equation*}
By the Borel-Cantelli lemma, $(\lambda_2^{-1}\lambda_j)^tY_{t,j}^{(i)} \to 0$ almost surely. 

\begin{proof}[Proof of Claim \ref{var_thm}]
	From Lemma \ref{E_general} and the fact that $\bC_0^{(i)}=(\bZ_0^{(i)})^\top \bZ_0^{(i)}$,
	\begin{equation*}
	\mathsf{Var}(\bZ_t^{(i)})=\bC_t^{(i)}-(\bM^\top)^t (\bZ_0^{(i)})^\top \bZ_0^{(i)} \bM^t=\sum_{l=1}^{t} (\bM^\top)^{t-l} \left(\sum_{k=1}^{N}\bV_k\EE Z_{l-1,k}^{(i)}\right) \bM^{t-l}.
	\end{equation*}
	As a result,
	\begin{equation*}
	\mathsf{Var}(\langle \bZ_t^{(i)}, \bbf_j \rangle)=\sum_{l=1}^{t} \bbf_j^\top (\bM^\top)^{t-l} \left(\sum_{k=1}^{N}\bV_k\EE Z_{l-1,k}^{(i)}\right) \bM^{t-l} \bbf_j.
	\end{equation*}
	Since $\bbf_j$ is the eigenvector of $\bM$ corresponding to the eigenvalue $m\lambda_j$, for every $n\in\ZZ^+$, $\bM^n \bbf_j=(m\lambda_j)^n \bbf_j$. This yields
	\begin{equation}
	\label{eq_innerproduct}
	\mathsf{Var}(\langle \bZ_t^{(i)}, \bbf_j \rangle)=\sum_{l=1}^{t} (m\lambda_j)^{2t-2l} \sum_{k=1}^{N} (\bbf_j^\top \bV_k \bbf_j) \EE Z_{l-1,k}^{(i)}.
	\end{equation}
	Notice that $\sum_{k=1}^{N} \EE Z_{l-1,k}=m^{l-1}$,
	\begin{equation*}
	\mathsf{Var}(\langle \bZ_t^{(i)}, \bbf_j \rangle) \leq c\sum_{l=1}^{t} (m\lambda_j)^{2t-2k} m^k=c(m\lambda_j)^{2t} \sum_{l=1}^{t} (m\lambda_j^2)^{-l},
	\end{equation*}
	where $c=\max\{\bbf_j^\top \bV_k \bbf_j:1\leq j,k\leq N\}$.
	This gives \eqref{eq_var}.
\end{proof}

\subsection{Proof of Lemma \ref{lem:nondegenerate-Y2}}
\label{subsec-lem-Y2}
By Lemma \ref{lem:martingale} and \ref{lem:martingale-convergence}, $\EE Y_2^{(i)}=Y_{0,2}^{(i)}=f_2(i)$ and
\begin{equation*}
\mathsf{Var}(Y_2^{(i)})=\lim_{t\to\infty}\mathsf{Var}(Y_{t,2}^{(i)})=\lim_{t\to\infty}(m\lambda_2)^{-2t}\mathsf{Var}(\langle \bZ_t^{(i)},\bbf_2 \rangle).
\end{equation*}
By \eqref{eq_innerproduct} 
\begin{equation*}
(m\lambda_2)^{-2t}\,\mathsf{Var}(\langle \bZ_t^{(i)}, \bbf_2 \rangle)=\sum_{l=1}^{t} (m\lambda_2)^{-2l} \sum_{k=1}^{N} (\bbf_2^\top \bV_k \bbf_2) \EE Z_{l-1,k}^{(i)}.
\end{equation*}
Notice that $\bV_k=m \left(\mathsf{diag}\{P_{k1},\ldots,P_{kN}\}- \bP_k \bP_k^T\right)$, where $\bP_k^\top=(P_{k1},\ldots,P_{kN})$ is the $k$-th row of $\bP$. Notice that $\sum_{j=1}^{N}P_{kj}=1$, by the Jensen's inequality
\begin{equation}
\label{eq:use_jensen}
\bbf_2^\top \bV_k \bbf_2= m \sum_{j=1}^{N} P_{kj}f_2(j)^2- m \left(\sum_{j=1}^{N} P_{kj} f_2(j)\right)^2\geq 0
\end{equation}
for any $k=1,\ldots,N$. The assumptions $\langle \bbf_1,\bbf_2\rangle_\pi=0$ and $\bbf_1=\textbf{1}$ imply that $f_2$ is not a constant vector, thus the equality in \eqref{eq:use_jensen} does not hold.  Similar to the proof of Theorem \ref{var_thm},
\begin{equation*}
(m\lambda_2)^{-2t}\,\mathsf{Var}(\langle \bZ_t^{(i)}, \bbf_2 \rangle) \geq c\sum_{l=1}^{t} (m\lambda_2)^{-2l} m^k=c\sum_{l=1}^{t} (m\lambda_2^2)^{-l}
\end{equation*}
where $c=\min\{\bbf_2^\top \bV_k \bbf_2:1\leq k\leq N\}>0$. Since $m\lambda_2^{2}>1$, this yields $\mathsf{Var}(Y_2^{(i)})>0$ for any $i=1,\ldots,N$.

\subsection{Proof of Corollary \ref{cor_decay}}
\label{subsec-proof-cor-decay}
	The $L^2$ convergence in Theorem \ref{thm_mut} implies $L^1$ convergence. If a sequence of random variables $X_n\xrightarrow{L^1} X$, then $\abs{\EE(X_n-X)}\leq \EE\abs{X_n-X}$ implies $\EE X_n \to \EE X$. So
	\begin{equation*}
	\lim_{t\to\infty}\EE\left(\lambda_2^{-t}[\hat{\mu}_t^{(i)}-\EE_{\bpi} (y)]\right) = \EE X^{(i)}=\frac{(m-1)\lambda_2}{m\lambda_2-1} \langle \by,\bbf_2 \rangle_{\bpi} f_2(i).
	\end{equation*}
	Since $\langle \by,\bbf_2 \rangle_{\bpi} f_2(i)\neq 0$, the bias term decays like
	\begin{equation*}
	\left(\EE (\hat{\mu}_t^{(i)})-\EE_{\bpi} (y)\right)^2=\Theta(\lambda_2^{2t}).
	\end{equation*}
	Additionally, the $L^2$ convergence in Theorem \ref{thm_mut} also yields
	$$\lambda_2^{-2t}\,\mathsf{Var}(\hat{\mu}_t^{(i)})=\mathsf{Var} (\lambda_2^{-t}[\hat{\mu}_t^{(i)}-\EE_{\bpi} (y)]) \xrightarrow{t\to\infty} \mathsf{Var}(X^{(i)})>0.$$
	So the variance term decays like $\mathsf{Var}(\hat{\mu}_t^{(i)})=\Theta(\lambda_2^{2t})$.

\subsection{Proof of Corollary \ref{cor_VH}}
	By the definition of the VH estimator in Section \ref{sec_ipw},
	$$\hat{\mu}^{(i)}_{\mathsf{VH},t}= H_t \cdot \frac{1}{n_t}   \sum_{\sigma \in \T, \abs{\sigma}\leq t} \frac{y(X^{(i)}_\sigma)}{\mathrm{deg}(X^{(i)}_\sigma)}, \quad \text{where} \ H_t^{-1} = \frac{1}{n_t} \sum_{\sigma \in \T} \frac{1}{\mathrm{deg}(X_{\sigma}^{(i)})}.$$
	$H_t^{-1}$ is the sample average of $y^\prime(X^{(i)}_\sigma)$'s up to generation $t$, where $y^\prime(j)=\mathrm{deg}(j)^{-1}$. In view of Theorem \ref{thm_mut}, $H_t^{-1}$ converges to $\EE_{\bpi} (y^\prime)>0$ almost surely. Additionally, 
	$$\hat{\mu}_t^{\prime\prime} =\frac{1}{n_t}   \sum_{\sigma \in \T, \abs{\sigma}\leq t} \frac{y(X_\sigma^{(i)})}{\mathrm{deg}(X_\sigma^{(i)})}$$
	is the sample average of $y^{\prime\prime}(X_\sigma^{(i)})$'s up to generation $t$, where $y^{\prime\prime}(j)=y(j)/\mathrm{deg}(j)$. By Theorem \ref{thm_mut}, there exists some random variable $\bar{X}^{(i)}\in L^2$ such that $\lambda_2^{-t} [\hat{\mu}_t ^{\prime\prime}-\EE_{\bpi} (y^{\prime\prime})]\to \bar{X}^{(i)}$ almost surely and in $L^2$. So
	$$\lambda_2^{-t}\left[\hat{\mu}^{(i)}_{\mathsf{VH},t}-\frac{\EE_{\bpi}(y^{\prime\prime})}{\EE_{\bpi}(y^\prime)}\right] \to \EE_{\bpi} (y^\prime)^{-1} \bar{X}^{(i)}\triangleq \tilde{X}^{(i)}$$
	almost surely. Notice that $\EE_{\bpi}(y^\prime)=N/\mathrm{vol}(G)$ and $\EE_{\bpi}(y^{\prime\prime})=\sum_i y(i)/\mathrm{vol}(G)$, this gives the result that 
	$$\lambda_2^{-t}\left[\hat{\mu}^{(i)}_{\mathsf{VH},t}-\mu_{\mathsf{true}}\right]\to \tilde{X}^{(i)}$$
	almost surely. The mean and variance of $\tilde{X}^{(i)}$ comes directly from Theorem \ref{thm_mut}.

\section{Proofs: GLS estimator}
\label{S:app-GLS}
This section contains the proof of Lemma \ref{lem:trivial'} and Corollary \ref{cor_glsVH}.

\subsection{Proof of Lemma \ref{lem:trivial'}}
\label{subsec-lem-trivial'}
	For a given $n$, there exists $t$ such that $n_{t-1} \leq n < n_t$. Throughout this proof, $t$ is determined by the corresponding $n$ in this way. 
	
	We consider two cases. First, when $n< n_{t-1}+m^{t-1}$, in base $m$, $n_t-n$ is represented as 
	\begin{equation}
	\label{eq:mbase1}
	n_t-n=a_{t-1} m^{t-1}+\cdots+a_{1} m +a_0,
	\end{equation}
	where $a_i\in\{0,1,\ldots,m-1\}$ for $0\leq i \leq t-1$, $a_{t-1}\geq 1$.
	And $\hat{\mu}_n^{(\nu)}$ can be represented as	
	\begin{equation}
	\label{eq:lemb2_1}
	\hat{\mu}_n^{(\bnu)}=\frac{n_{t}{\hat{\mu}_t^{(\bnu)}}-\sum_{k=n+1}^{n_t}y(X_k^{(\bnu)})}{n}.
	\end{equation}
	Note that $\{X_k^{(\bnu)}:n_t-m^{t-1}+1\leq k \leq n_t\}$ form the $(t-1)$-st generation of a subtree of $\T$ (rooted at a child of the root $\T$) and let $W_{t-1}^1=\sum_{k=n_t-m^{t-1}+1}^{n_t} y(X_k^{(\nu)})$. Similarly we can determine $a_{t-1}$ such subtrees by scanning the nodes from right to left in the $t$-th generation of $\T$ and define accordingly $W_{t-1}^2,\ldots,W_{t-1}^{a_{t-1}}$. Next we can determine a subtree of $\T$ where the next $m^{t-2}$ nodes in the $t$-th generation of $\T$ form its $(t-2)$-nd generation. We can determine $a_{t-2}$ such subtrees by continuing to scan the nodes from right to left in the $t$-th generation of $\T$ and define $W_{t-2}^1,\ldots,W_{t-2}^{a_{t-2}}$. 
	And so on.
	By \eqref{eq:mbase1},
	\begin{equation*}
	\sum_{k=n+1}^{n_t} y(X_k^{(\bnu)}) = \sum_{k=1}^{a_{t-1}}W_{t-1}^{k}+\sum_{k=1}^{a_{t-2}}W_{t-2}^{k}+\cdots+\sum_{k=1}^{a_{0}}W_{0}^{k}.
	\end{equation*}
	To proceed, we need the following concentration bounds for $m^{-t}W_t^{(\bnu)}$ and $\hat{\mu}_t^{(\bnu)}$ (proof below). 
	\begin{claim}
		\label{lem:trivial}
		For any initial distribution $\bnu$ of $X_0$, $m^{-t}W_t^{(\bnu)}\to \EE_{\bpi}(y)$ and $\hat{\mu}_t^{(\bnu)} \to \EE_{\bpi}(y)$ in $L^2$. For any $0<\delta<1$, there exists $C>0$ such that
		\begin{equation}
		\label{eq:lemb1rate}
		\EE[(m^{-t}W_t^{(\bnu)}-\EE_{\bpi}(y))^2] \leq Cm^{-(1-\delta)t}, \quad
		\EE[(\hat{\mu}_t^{(\bnu)}-\EE_{\bpi} (y))^2]\leq Ctm^{-(1-\delta)t}.
		\end{equation}
		The constant $C$ does not depend on the initial distribution $\nu$.
	\end{claim}
	Then by Claim \ref{lem:trivial}, the triangle inequality, $a_{t-1}\geq 1$ and $a_l\leq m-1$ for $0\leq l \leq t-1$, one has
	\begin{equation}
	\label{eq:lemb2_big}
	\begin{aligned}
	\left\Vert\frac{\sum_{k=n+1}^{n_t} y(X_k^{(\bnu)})}{n_t-n}-\EE_{\bpi}(y)\right\Vert_{L^2}&=\left\Vert\frac{\sum_{l=0}^{t-1}m^{l}\sum_{k=1}^{a_l}m^{-l}(W_{l}^{k}-\EE_{\bpi}(y))}{a_{t-1} m^{t-1}+\cdots+a_{1} m +a_0}\right\Vert_{L^2} \\
	&\leq \sum_{l=0}^{t-1}\frac{a_l m^l}{a_{t-1} m^{t-1}+\cdots+a_{1} m +a_0}  Cm^{-\frac{1-\delta}{2}l}\\
	&\leq C \sum_{l=0}^{t-1}\frac{(m-1) m^l}{m^{t-1}} m^{-\frac{1-\delta}{2}l} = O(m^{-\frac{1-\delta}{2}t}).
	\end{aligned}
	\end{equation}
	For any subsequence such that $n<n_{t-1}+m^{t-1}$, $n\to\infty$ implies $t\to \infty$. As a result,
	\begin{equation*}
	\frac{\sum_{k=n+1}^{n_t} y(X_k^{(\bnu)})}{n_t-n}\xrightarrow{L^2}\EE_{\bpi}(y).
	\end{equation*}
	From Claim \ref{lem:trivial}, $\hat{\mu}_t^{(\bnu)} \xrightarrow{L^2} \EE_{\bpi}(y)$. By \eqref{eq:lemb2_1}, the triangle inequality, the fact that $n_t/n=O(1)$ and $(n_t-n)/n=O(1)$, 
	\begin{equation*}
	\left\Vert\hat{\mu}_n^{(\nu)}-\EE_{\bpi}(y)\right\Vert_{L^2}\leq \frac{n_t}{n}\left\Vert \hat{\mu}_t^{(\nu)}-\EE_{\bpi}(y)\right\Vert_{L^2}+\frac{n_t-n}{n}\left\Vert \frac{\sum_{k=n+1}^{n_t} y(X_k^{(\nu)})}{n_t-n}-\EE_{\bpi}(y)\right\Vert_{L^2}\xrightarrow{n\to\infty} 0.
	\end{equation*}
	
	In the second case, when $n\geq n_{t-1}+m^{t-1}$, in base $m$, $n-n_{t-1}$ is represented as 
	\begin{equation}
	\label{eq:mbase2}
	n-n_{t-1}=a_{t-1} m^{t-1}+\cdots+a_{1} m +a_0,
	\end{equation}
	where $a_i\in\{0,1,\ldots,m-1\}$ for $0\leq i \leq t-1$, $a_{t-1}\geq 1$.
	And $\hat{\mu}_n^{(\bnu)}$ can be represented as	
	\begin{equation}
	\label{eq:lemb2_2}
	\hat{\mu}_n^{(\bnu)}=\frac{n_{t-1}{\hat{\mu}_{t-1}^{(\bnu)}}+\sum_{k=n_{t-1}+1}^{n}y(X_k^{(\bnu)})}{n}.
	\end{equation}
	Arguing as above, we can write
	\begin{equation*}
	\sum_{k=n_{t-1}+1}^{n} y(X_k^{(\bnu)}) = \sum_{k=1}^{a_{t-1}}W_{t-1}^{k}+\sum_{k=1}^{a_{t-2}}W_{t-2}^{k}+\cdots+\sum_{k=1}^{a_{0}}W_{0}^{k}.
	\end{equation*}
	Similarly to the previous case, we can prove that for any subsequence such that $n\geq n_{t-1}+m^{t-1}$, when $n\to\infty$,
	\begin{equation*}
	\frac{\sum_{k=n_{t-1}+1}^{n} y(X_k^{(\bnu)})}{n-n_{t-1}}\xrightarrow{L^2}\EE_{\bpi}(y),
	\end{equation*}
	and
	\begin{equation*}
	\left\Vert\hat{\mu}_n^{(\bnu)}-\EE_{\bpi}(y)\right\Vert_{L^2}\leq \frac{n_{t-1}}{n}\left\Vert \hat{\mu}_{t-1}^{(\bnu)}-\EE_{\bpi}(y)\right\Vert_{L^2}+\frac{n-n_{t-1}}{n}\left\Vert \frac{\sum_{k=n_{t-1}+1}^{n} y(X_k^{(\nu)})}{n-n_{t-1}}-\EE_{\bpi}(y)\right\Vert_{L^2}\xrightarrow{n\to\infty} 0.
	\end{equation*}
	
	Since $\hat{\mu}_n^{(\bnu)}\xrightarrow{L^2}\EE_{\bpi}(y)$ holds for both $n<n_{t-1}+m^{t-1}$ and $n\geq n_{t-1}+m^{t-1}$ as $n\to\infty$, one finally arrives at $\hat{\mu}_n^{(\bnu)}\xrightarrow{L^2}\EE_{\bpi}(y)$, which completes the proof.
	
	\begin{proof}[Proof of Claim \ref{lem:trivial}]
		The proof is similar to the proof of Theorem \ref{thm_mut}. First,
		\begin{equation*}
		\EE [(\lambda_2^t Y_{t,2}^{(i)})^2]=\mathsf{Var}(\lambda_2^t Y_{t,2}^{(i)})+(\lambda_2^t\EE Y_{t,2}^{(i)})^2=m^{-2t}\mathsf{Var}(\langle \bZ_t^{(i)}, \bbf_2 \rangle)+(\lambda_2^tY_{0,2}^{(i)})^2.
		\end{equation*}
		From \eqref{eq_var}, for any $0<\delta<1$, $\mathsf{Var}(\langle \bZ_t^{(i)}, \bbf_2 \rangle)=O(m^{(1+\delta)t})$ holds for all $i\in G$. As a result,
		\begin{equation}
		\EE [(\lambda_2^t Y_{t,2}^{(i)})^2]=O(m^{-(1-\delta)t}).
		\end{equation}
		From \eqref{eq_Wt}, one has $m^{-t}W_t^{(i)}-\EE_{\bpi} (y)=\langle \by,\bbf_2 \rangle_{\bpi} \lambda_2^t Y_{t,2}^{(i)}$ and as a result,
		\begin{equation}
		\label{eq:Wl2}
		\EE[(m^{-t}W_t^{(i)}-\EE_{\bpi}(y))^2]\leq \langle \by,\bbf_2\rangle_{\bpi}^2 	\EE [(\lambda_2^t Y_{t,2}^{(i)})^2]=O(m^{-(1-\delta)t}).
		\end{equation}
		Recall from \eqref{eq_as_converge} that
		\begin{equation*}
		\hat{\mu}_t^{(i)}-\EE_{\bpi} (y)=\frac{m^t}{n_t} \sum_{l=0}^{t}\frac{W_l^{(i)}-m^l \EE_{\bpi}(y)}{m^t}.
		\end{equation*}
		By the Cauchy-Schwarz inequality,
		\begin{equation}
		\label{eq:mul2}
		\begin{aligned}
		\EE\left[(\hat{\mu}_t^{(i)}-\EE_{\bpi} (y))^2\right]&\leq \left(\frac{m^t}{n_t}\right)^2 (t+1)\sum_{l=0}^{t} m^{2(l-t)}\EE[(m^{-l}W_t^{(i)}-\EE_{\bpi}(y))^2]\\
		&=\left(\frac{m^t}{n_t}\right)^2 (t+1) \sum_{l=0}^{t} O(m^{-2t+(1+\delta)l}) \\
		&=O(tm^{-(1-\delta)t}).
		\end{aligned}
		\end{equation}
		So there exists $C>0$ such that for all $i\in G$, $$\EE[(m^{-t}W_t^{(i)}-\EE_{\bpi}(y))^2]\leq Cm^{-(1-\delta)t}, \quad \EE[(\hat{\mu}_t^{(i)}-\EE_{\bpi} (y))^2]\leq Ctm^{-(1-\delta)t}.$$
		So for any initial distribution $\bnu$ of $X_0$, since $\sum_{i\in G}\nu_i=1$,
		\begin{equation*}
		\EE[(\hat{\mu}_t^{(\bnu)}-\EE_{\bpi} (y))^2]=\sum_{i\in G} \nu_i \EE[(\hat{\mu}_t^{(i)}-\EE_{\bpi} (y))^2]\leq Ctm^{(1-\delta)t},
		\end{equation*}
		\begin{equation*}
		\EE[(m^{-t}W_t^{(\bnu)}-\EE_{\bpi}(y))^2] =\sum_{i\in G} \nu_i \EE[(m^{-t}W_t^{(i)}-\EE_{\bpi}(y))^2]\leq Cm^{-(1-\delta)t}.
		\end{equation*}
		So $m^{-t}W_t^{(\bnu)}\to \EE_{\bpi}(y)$ and $\hat{\mu}_t^{(\bnu)} \to \EE_{\bpi}(y)$ in $L^2$. 
	\end{proof}

\subsection{Proof of Corollary \ref{cor_glsVH}}
	By the definition of the GLS estimator with VH adjustment in Section \ref{sec_ipw},
	$$\hat{\mu}_{\mathsf{GLS,VH},t}^{(\bnu)}= H_t\cdot \sum_{\sigma \in \T, \abs{\sigma} \leq t} w^*_{\sigma,t} \frac{y(X_\sigma^{(\nu)})}{\mathrm{deg}(X_\sigma^{(\nu)})}.$$
	$H_t^{-1}$ is the GLS estimator of $\EE_{\bpi}(y^\prime)$ where $y^\prime(i)=\mathrm{deg}(i)^{-1}$. So $H_t^{-1}$ converges to $\EE_{\bpi} (y^\prime)$ in distribution (thus in probability).
	Additionally, 
	$$\hat{\mu}^{\prime\prime}_{\mathsf{GLS},t}=   \sum_{\sigma \in \T, \abs{\sigma} \leq t} w^*_{\sigma,t} \frac{y(X_\sigma^{(\bnu)})}{\mathrm{deg}(X_\sigma^{(\bnu)})}$$
	is the GLS estimator of $\EE_{\bpi} (y^{\prime\prime})$ where $y^{\prime\prime}(i)=y(i)/\mathrm{deg}(i)$. Then
	$$\sqrt{n_t} \left[\hat{\mu}^{\prime\prime}_{\mathsf{GLS},t}-\EE_{\bpi}(y^{\prime\prime})\right]\xrightarrow{d} \cN\left(0,\frac{1+\lambda_2}{1-\lambda_2}\mathsf{Var}_{\bpi}(y^{\prime\prime})\right).$$ 
	By Slutsky's theorem,
	$$\sqrt{n_t}\left[\hat{\mu}_{\mathsf{GLS,VH},t}^{(\bnu)}-\frac{\EE_{\bpi}(y^{\prime\prime})}{\EE_{\bpi}(y^\prime)}\right] \xrightarrow{d} \cN\left(0,\frac{1+\lambda_2}{1-\lambda_2}\EE_{\bpi}(y^\prime)^{-2}\mathsf{Var}_{\bpi}(y^{\prime\prime})\right).$$
	Notice that $\EE_{\bpi}(y^\prime)=N/\mathrm{vol}(G)$ and $\EE_{\bpi}(y^{\prime\prime})=\sum_i y(i)/\mathrm{vol}(G)$, this gives the result
	$$\sqrt{n_t}\left[\hat{\mu}_{\mathsf{GLS,VH},t}^{(\bnu)}-\mu_{\mathsf{true}}\right] \xrightarrow{d} \cN\left(0,\frac{1+\lambda_2}{1-\lambda_2}\EE_{\bpi}(y^\prime)^{-2}\mathsf{Var}_{\bpi}(y^{\prime\prime})\right).$$

\bibliographystyle{ims}
\bibliography{reference}

\begin{thebibliography}{20}
\expandafter\ifx\csname natexlab\endcsname\relax\def\natexlab#1{#1}\fi
\expandafter\ifx\csname url\endcsname\relax
  \def\url#1{\texttt{#1}}\fi
\expandafter\ifx\csname urlprefix\endcsname\relax\def\urlprefix{URL }\fi

\bibitem[{Athreya and Ney(2004)}]{athreya2004branching}
\textsc{Athreya, K.~B.} and \textsc{Ney, P.~E.} (2004).
\newblock \textit{Branching processes}.
\newblock Courier Corporation.

\bibitem[{Baraff et~al.(2016)Baraff, McCormick and
  Raftery}]{baraff2016estimating}
\textsc{Baraff, A.~J.}, \textsc{McCormick, T.~H.} and \textsc{Raftery, A.~E.}
  (2016).
\newblock Estimating uncertainty in respondent-driven sampling using a tree
  bootstrap method.
\newblock \textit{Proceedings of the National Academy of Sciences}  201617258.

\bibitem[{Benjamini and Peres(1994)}]{benjamini1994markov}
\textsc{Benjamini, I.} and \textsc{Peres, Y.} (1994).
\newblock Markov chains indexed by trees.
\newblock \textit{The Annals of Probability}  219--243.

\bibitem[{CDC(2017)}]{HIVbehavioralSurveilance}
\textsc{CDC} (2017).
\newblock {National HIV Behavioral Surveillance (NHBS)}.
\newblock \textit{Division of HIV/AIDS Prevention} .

\bibitem[{Durrett(2019)}]{durrett2019probability}
\textsc{Durrett, R.} (2019).
\newblock \textit{Probability: theory and examples}, vol.~49.
\newblock Cambridge university press.

\bibitem[{Goel and Salganik(2009)}]{goel2009respondent}
\textsc{Goel, S.} and \textsc{Salganik, M.~J.} (2009).
\newblock Respondent-driven sampling as {M}arkov chain {M}onte {C}arlo.
\newblock \textit{Statistics in medicine} \textbf{28} 2202--2229.

\bibitem[{Harris(2011)}]{harris2011national}
\textsc{Harris, K.~M.} (2011).
\newblock The national longitudinal study of adolescent health: Research
  design.
\newblock \textit{http://www. cpc. unc. edu/projects/addhealth/design} .

\bibitem[{Harris(2002)}]{harris2002theory}
\textsc{Harris, T.~E.} (2002).
\newblock \textit{The theory of branching processes}.
\newblock Courier Corporation.

\bibitem[{Heckathorn(1997)}]{heckathorn1997respondent}
\textsc{Heckathorn, D.~D.} (1997).
\newblock Respondent-driven sampling: a new approach to the study of hidden
  populations.
\newblock \textit{Social problems} \textbf{44} 174--199.

\bibitem[{Holland and Laskey(1983)}]{Holland1983Stochastic}
\textsc{Holland, P.~W.} and \textsc{Laskey, K.~B.} (1983).
\newblock Stochastic blockmodels: First steps.
\newblock \textit{Social Networks} \textbf{5} 109--137.

\bibitem[{Johnston(2013)}]{johnston2013introduction}
\textsc{Johnston, L.} (2013).
\newblock Introduction to hiv/aids and sexually transmitted infection
  surveillance: Module 4: Introduction to respondent driven sampling.
\newblock \textit{World Health Organization} .

\bibitem[{Kesten and Stigum(1966)}]{kesten1966additional}
\textsc{Kesten, H.} and \textsc{Stigum, B.~P.} (1966).
\newblock Additional limit theorems for indecomposable multidimensional
  galton-watson processes.
\newblock \textit{The Annals of Mathematical Statistics} \textbf{37}
  1463--1481.

\bibitem[{Levin et~al.(2009)Levin, Peres and Wilmer}]{levin2009markov}
\textsc{Levin, D.~A.}, \textsc{Peres, Y.} and \textsc{Wilmer, E.~L.} (2009).
\newblock \textit{Markov chains and mixing times}.
\newblock American Mathematical Soc.

\bibitem[{Li and Rohe(2017)}]{li2017central}
\textsc{Li, X.} and \textsc{Rohe, K.} (2017).
\newblock Central limit theorems for network driven sampling.
\newblock \textit{Electronic Journal of Statistics} \textbf{11} 4871--4895.

\bibitem[{Roch and Rohe(2018)}]{roch2018generalized}
\textsc{Roch, S.} and \textsc{Rohe, K.} (2018).
\newblock Generalized least squares can overcome the critical threshold in
  respondent-driven sampling.
\newblock \textit{Proceedings of the National Academy of Sciences} \textbf{115}
  10299--10304.

\bibitem[{Rohe(2019)}]{rohe2019critical}
\textsc{Rohe, K.} (2019).
\newblock A critical threshold for design effects in network sampling.
\newblock \textit{The Annals of Statistics} \textbf{47} 556--582.

\bibitem[{Salganik and Heckathorn(2004)}]{salganik2004sampling}
\textsc{Salganik, M.~J.} and \textsc{Heckathorn, D.~D.} (2004).
\newblock Sampling and estimation in hidden populations using respondent-driven
  sampling.
\newblock \textit{Sociological methodology} \textbf{34} 193--240.

\bibitem[{Volz and Heckathorn(2008)}]{volz2008probability}
\textsc{Volz, E.} and \textsc{Heckathorn, D.~D.} (2008).
\newblock Probability based estimation theory for respondent driven sampling.
\newblock \textit{Journal of official statistics} \textbf{24} 79.

\bibitem[{White et~al.(1976)White, Boorman and Breiger}]{White1976Social}
\textsc{White, H.~C.}, \textsc{Boorman, S.~A.} and \textsc{Breiger, R.~L.}
  (1976).
\newblock Social structure from multiple networks. i. blockmodels of roles and
  positions.
\newblock \textit{American Journal of Sociology} \textbf{81} 730--780.

\bibitem[{White et~al.(2015)White, Hakim, Salganik, Spiller, Johnston, Kerr,
  Kendall, Drake, Wilson, Orroth et~al.}]{white2015strengthening}
\textsc{White, R.~G.}, \textsc{Hakim, A.~J.}, \textsc{Salganik, M.~J.},
  \textsc{Spiller, M.~W.}, \textsc{Johnston, L.~G.}, \textsc{Kerr, L.},
  \textsc{Kendall, C.}, \textsc{Drake, A.}, \textsc{Wilson, D.},
  \textsc{Orroth, K.} \textsc{et~al.} (2015).
\newblock Strengthening the reporting of observational studies in epidemiology
  for respondent-driven sampling studies:“strobe-rds” statement.
\newblock \textit{Journal of clinical epidemiology} \textbf{68} 1463--1471.

\end{thebibliography}

\newpage

\end{document}